\documentclass[%
 aip,
 amsmath,amssymb,
 reprint,%
]{revtex4-1}

\usepackage{graphicx}% Include figure files
\usepackage{dcolumn}% Align table columns on decimal point
\usepackage{bm}% bold math
\usepackage[utf8]{inputenc}
\usepackage[T1]{fontenc}
\usepackage{mathptmx}
\usepackage{etoolbox}
\usepackage{xcolor}

\newtheorem{theorem}{Theorem}

\newtheorem{definition}[theorem]{Definition}
\newtheorem{example}[theorem]{Example}

\newtheorem{proposition}[theorem]{Proposition}
\newtheorem{remark}[theorem]{Remark}

\newenvironment{proof}[1][Proof]{\noindent\textbf{#1.} }{\ \rule{0.5em}{0.5em}}
\makeatletter
\def\@email#1#2{%
 \endgroup
 \patchcmd{\titleblock@produce}
  {\frontmatter@RRAPformat}
  {\frontmatter@RRAPformat{\produce@RRAP{*#1\href{mailto:#2}{#2}}}\frontmatter@RRAPformat}
  {}{}
}%
\makeatother
\begin{document}

\preprint{AIP/123-QED}

%\title[Sample title]{}
\title{Entropic and algebraic transcript-based tools
in time series analysis}
% Force line breaks with \\
\author{José M. Amigó}
% \altaffiliation[Also at ]{Centro de Investigaci\'{o}n Operativa, Universidad Miguel Hern\'{a}ndez, 03202 Elche, Spain}%Lines break automatically or can be forced with \\
\author{Roberto Dale}%
 \email{jm.amigo@umh.es, rdale@umh.es.}
\affiliation{Centro de Investigaci\'{o}n Operativa, Universidad Miguel Hern\'{a}ndez, 03202 Elche, Spain} 
%\\This line break forced with \textbackslash\textbackslash

%\homepage{http://www.Second.institution.edu/~Charlie.Author.}
%\affiliation{%
%Second institution and/or address%\\This line break forced% with \\
%}%

\date{\today}% It is always \today, today,

\begin{abstract}

Algebraic representations of time series are symbolic representations whose
symbols belong to a finite group. Precisely, the framework of the present
paper is the analysis of coupled time series in algebraic representations
and, more generally, group-valued time series. The prototype of an algebraic
representation is an ordinal representation, whose symbols are permutations,
also called ordinal patterns in the context of time series analysis. In
fact, permutations, endowed with function composition, build a group called
a symmetric group. A simple way to harness the algebraic structure of the
alphabet in such cases is the concept of transcript from one group element
to another. Since transcripts involve two group elements, they are very
suitable for studying couplings between time series in the same algebraic
representation. In this paper, we outline several existing entropic and algebraic 
transcript-based tools for analyzing coupled time series and systems.
In addition to entropy, the entropic tools include divergence, statistical
complexity and mutual information. The algebraic tools comprise order
classes and, most recently, the Cayley and Kendall distances. We use the detection of generalized synchronization in a well-studied coupled system to compare the performances of some of those tools. To this end,
we also provide an alternative tool called the similarity distance between
times series, which is a mean Kendall distance. We found that the novel similarity distance outperforms the other tools tested.
\end{abstract}

\maketitle

\begin{quotation}
Symbolization is a useful method to gain valuable
insights in time series analysis. To do this, the entries are generally
discretized and replaced with symbols. In the case of algebraic
representations (one of the main characters of this paper), the symbols make
up an algebraic structure called a group, which allows for additional
leverage. Such is the case when the symbols are permutations (also called
ordinal patterns), obtained by ranking the entries of a time series within
sliding windows. This symbolization technique has become very popular among
data analysts not only for its conceptual simplicity and practical
advantages, but also for its proven power in extracting useful information
from data. A simple way to exploit the algebraic properties of the symbols
in algebraic representations is the concept of transcript. In this paper, we
discuss the basics of transcripts, their connection with two algebraic
distances in both permutation and general groups, and the application of
transcript-based tools in the analysis of coupled time series.
\end{quotation}

%%%%%%%%%%%%%%%%%%%%%%%%%%%%%%%%%%%%%%%%%%%%%%%%%%%%%%%%%%%%%%%%%%%%%%%%%%%%
\section{Introduction}

\label{sec1}

Symbolic representation of time series is a usual technique in time series
analysis \cite{Kantz1997}, where by time series we mean here finite or
infinite sequences of real numbers (the data). The perhaps simplest instance
consists of partitioning the range of the data and labeling the resulting
intervals; the entries of the time series are then trade-off for the labels
of the intervals they belong to. This procedure is well known in statistics
and dynamical systems \cite{Hirata2023}. More general methods trade-off
segments of the time series (even the whole series if finite) for graphs,
topological properties or algebraic structures, such as different types of
networks \cite{Zhang2006,Small2013,Zou2019}, topological invariants \cite%
{Lum2013,Perea2019,Haruna2023} and finite groups \cite{Amigo2016}. In such
cases, we speak of graph-theoretical, topological and algebraic
representations, respectively. The prototype of the latter are the ordinal
representations, i.e., symbolic representations whose symbols are rank
vectors, usually referred to as ordinal patterns \cite{Bandt2002,Mohr2020}
in time series analysis. Therefore, ordinal patterns can be viewed as
permutations, which build a classic algebraic group called the symmetric group \cite%
{Herstein1996,Lang2005}.

In this paper we are going to focus on algebraic representations and how to
take advantage of the fact that their alphabets are groups. To the best of
our knowledge, the algebraic nature of ordinal patterns was first exploited
by Monetti et al. \cite{Monetti2009}, who introduced the concept of
transcript from one ordinal pattern to another; its generalization to
algebraic representations is straightforward \cite{Amigo2016}. Furthermore,
since transcripts involve two symbols, they are well-suited for studying
serial dependence in time series (self-transcripts) and the coupling between
systems (cross-transcripts) \cite{Amigo2012}. We are not aware of any other
proposals in this regard.

Therefore, here we will discuss the concept of transcript in a group,
present some traditional, recent and novel transcript-based tools in time
series analysis, and illustrate their applications with numerical
simulations. Among the traditional tools, we consider some entropic ones
(basically: entropy, divergence and mutual information) and the algebraic
order classes, all of which have a record of practical applications \cite%
{Monetti2009,Amigo2012}. The recent tools comprise two algebraic distances,
namely, the Cayley and Kendall distances \cite{Amigo2025}. And the novel
tool is a mean Kendall distance between two time series proposed in this
paper, which we call the similarity distance.

To compare the performance of the above-mentioned transcript-based tools, we use the detection of generalized synchronization
in a drive-response system. As it turns out, the performance of the
similarity distance compares favorably with that of the entropy-complexity
plane, transcript mutual information, transcript order classes, and the
transcript probability distribution-based Jensen-Shannon distance (the
square root of the Jensen-Shannon divergence). This suggests that the
similarity distance can be a useful tool in time series analysis.

Interestingly, as we transition from the concept of transcript to its
applications in time series analysis, we will come across direct connections
between transcripts and basic results in group theory, a very rewarding fact
because transcripts are a hallmark of algebraic representations. Indeed, the
transcript from one group element to another can be used to calculate the
Cayley and Kendall distances between them. Moreover, Cayley's theorem, which
embeds any finite group into a certain symmetric group, can be implemented
via transcripts. Incidentally, Cayley's embedding shows that ordinal
patterns are a kind of universal symbols for algebraic representations,
although this encoding is practical only for small groups \cite{Amigo2025}.

To sum up, this paper has three main objectives.

\begin{enumerate}
\item[(1)] To provide an overview of transcripts and their applications in
algebraic representations of time series;

\item[(2)] To compare the performances of several transcript-based tools
using numerical simulations;

\item[(3)] To present a novel algebraic tool, the similarity distance (a mean Kendall distance), which outperforms the other options in our numerical
tests.
\end{enumerate}

With these objectives in mind, this paper is organized as follows. We begin
by introducing the main characters: permutation groups (Section \ref{sec2}),
algebraic representations of time series (Section \ref{sec3}), and the
concept of transcript (Section \ref{sec4}). The latter will then allow us to
implement a homomorphism from any finite group to a group of permutations
(Cayley's theorem). In Section \ref{sec5} we outline a selection of
entropic and algebraic transcript-based tools for time series analysis. This
selection ranges from the more traditional ones (entropy, entropy-like
quantities and order classes) to the more recent ones (the Cayley and
Kendall distances). The numerical simulations of Section \ref{sec7} using
two non-identical, unidirectionally coupled H\'{e}non maps aim to illustrate
the potential of the afore-mentioned tools in applications, in this case,
for detecting generalized synchronization. One of these tools is the
similarity distance, proposed in Section \ref{sec74}, which discriminates
between strong and weak generalized synchronizations better than its
competitors. The conclusions, together with a summary of the main contributions of this paper, are the content of
Section \ref{sec8}.

%%%%%%%%%%%%%%%%%%%%%%%%%%%%%%%%%%%%%%%%%%%%%%%%%%%

\section{Ordinal representations and the groups of permutations}

\label{sec2}

The concept of algebraic representation in time series analysis is central
to this article. In this section we present its prototype, the ordinal
representation, a symbolic representation of time series whose symbols are
permutations, also called ordinal patterns, which was proposed by Bandt and
Pompe \cite{Bandt2002}.

Let $\mathrm{x}=(x_{t})_{t\geq 0}$ be an $\mathbb{R}$-valued time series.
The \textit{ordinal representation} of $\mathrm{x}$ with parameter $L\geq 2$
is the symbolic time series $(\mathbf{r}_{t})_{t\geq 0}$, where $\mathbf{r}%
_{t}=(r_{1},r_{2},...,r_{L})$ is the rank vector of the sequence (window,
block, word, ...) $x_{t}^{L}:=x_{t},x_{t+1},...,x_{t+L-1}$, i.e., 
\begin{equation}
x_{t+r_{1}-1}<x_{t+r_{2}-1}<\ldots <x_{t+r_{L}-1}.  \label{ord patt}
\end{equation}%
We will write $\mathbf{r}_{t}=\mathrm{rank}(x_{t}^{L})$ and use the set $%
\{1,2,...,L\}$ to rank the elements of $x_{t}^{L}$. For example, if $%
x_{t}^{3}=1.7,$\thinspace $0.5,\,1.2$, then $\mathbf{r}_{t}=(2,3,1)$. In
case of a tie $x_{t+i}=x_{t+j}$, there are several conventions, e.g., $%
x_{t+i}<x_{t+j}$ if $i<j$. The perhaps most popular convention (especially
if there are many ties) is to add random noise to eliminate the tie \cite%
{Myers2020}. In addition to equation (\ref{ord patt}), other definitions of
ordinal patterns can be found in the literature \cite{Bian2012,Unakafova2013,Schnurr2022}.

In equation (\ref{ord patt}) it is assumed that the delay time\ \textrm{T}
used to analyze the time series $\mathrm{x}$ is $1$. In practice, though,
when the rate of change of the data is small compared to the sampling
frequency of the signal, it may be convenient to use $\mathrm{T}>1$, in
which case equation (\ref{ord patt}) generalizes to%
\begin{equation}
x_{t+r_{1}\mathrm{T}-1}<x_{t+r_{2}\mathrm{T}-1}<\ldots <x_{t+r_{L}\mathrm{T}%
-1}.  \label{ord patt 2}
\end{equation}%
Therefore, ordinal representations have two parameters: the length of the
patterns $L\geq 2$ (sometimes called embedding dimension) and the delay time 
$\mathrm{T}\geq 1$ (also called delay lag). Here $\mathrm{T}=1$, unless otherwise
stated. An overview of selection techniques for delay times can be found in Tan et al. \cite{Tan2023}.

The rank vectors $\mathbf{r}_{t}=\mathrm{rank}(x_{t}^{L})$ are usually
called \textit{ordinal patterns} of length $L$, ordinal $L$-patterns, or
simply $L$-patterns if the only patterns in question are the ordinal ones.
They are also aptly called permutations since, as any total ranking, $%
\mathbf{r}_{t}=(r_{1},r_{2},...,r_{L})$ can be viewed as the one-line form
of the permutation%
\begin{equation}
\left( 
\begin{array}{ccccc}
1 & \cdots & k & \cdots & L \\ 
r_{1} & \cdots & r_{k} & \cdots & r_{L}%
\end{array}%
\right) .  \label{two-line form}
\end{equation}%
This is the view that best suits our purposes. For this reason, we will
mainly use the one-line form $(r_{1},r_{2},...,r_{L})$ and only occasionally
the two-line form (\ref{two-line form}) for permutations. In numerical
examples, the permutation $(r_{1},r_{2},...,r_{L})$ will be shortened to $%
r_{1}r_{2}...r_{L}$; for example, the permutation $(2,3,1)$ will be written
as $231$.

\begin{remark}
\label{RemarkAlternativeSymbols}Other symbols for the rank relationships (%
\ref{ord patt}) and (\ref{ord patt 2}) have also been used in the
literature, e.g., 0-1 square matrices by Bunk et al. \cite{Bunk2013} and
Haruna \cite{Haruna2019}, antisymmetric matrices with off-diagonal
components $\pm 1$ by Rios de Souza and Hlinka \cite{RiosSouza2022}, and the
so-called \textquotedblleft rank sequences\textquotedblright\ by Haruna and
Nakajima \cite{Haruna2011}.
\end{remark}

Ordinal patterns have been used for classification \cite%
{Keller2003,Parlitz2012}, characterization of dynamics and couplings \cite%
{Bandt2002,Monetti2009,Parlitz2013}, tests of serial dependence \cite%
{Canovas2011,Weiss2022,Weiss2022B}, and even cryptanalysis \cite%
{Arroyo2009}. A generalization of ordinal patterns to multivariate data
was proposed by Mohr et al. \cite{Mohr2020}.

Regarding the objectives of this paper, the most important feature of the
ordinal representation $(\mathbf{r}_{t})_{t\geq 0}$ is the fact that the
symbols $\mathbf{r}_{t}$ belong to a group. In general, a group $(\mathcal{G}%
,\cdot )$ is a set $\mathcal{G}$ endowed with a binary operation
\textquotedblleft $\cdot $\textquotedblright , sometimes called composition
law or product, satisfying the following properties \cite%
{Lang2005,Fraleigh2013}.

\begin{description}
\item[(G1)] \textit{Associativity}: For all $a,b,c\in \mathcal{G}$, it is
true that $(a\cdot b)\cdot c=a\cdot (b\cdot c)$.

\item[(G2)] \textit{Identity element}: There exists an element $e\in 
\mathcal{G}$, called the \textit{unity} (or \textit{neutral}) \textit{element%
}, such that $a\cdot e=e\cdot a=a$ for all $a\in \mathcal{G}$.

\item[(G3)] \textit{Inverse element}: For every $a\in \mathcal{G}$, there
exists an element $a^{-1}\in \mathcal{G}$, called the \textit{inverse element%
} of $a$, such that $a\cdot a^{-1}=a^{-1}\cdot a=e$.
\end{description}

It can be proved that the identity of a group and the inverse of each
element are unique. Groups whose product is commutative (i.e., $a\cdot
b=b\cdot a$ for all $a,b\in \mathcal{G}$) are called \textit{commutative} or 
\textit{abelian}. Examples of abelian groups are the real numbers endowed
with addition, and the nonzero real numbers endowed with multiplication;
examples of non-commutative groups will appear shortly. If the product is
clear from the context, then $(\mathcal{G},\cdot )$ is shortened to $%
\mathcal{G}$.

In the case of the \textit{ordinal representation} with parameter $L$, the
representation group is the \textit{symmetric group of degree} $L$, denoted $%
\mathrm{Sym}(S)$, where $\mathrm{Sym}(S)$ is composed of the permutations
(i.e., bijections) of a set $S$ with cardinality $\left\vert S\right\vert =L$
and the binary operation is the composition of permutations. Since the
properties of $\mathrm{Sym}(S)$ do not depend on $S$ but only on $\left\vert
S\right\vert $, we choose the conventional set $S=\{1,2,...,L\}$ (see
equation (\ref{two-line form})), unless otherwise stated, and also denote $%
\mathrm{Sym}(S)$ by $\mathrm{Sym}(\left\vert S\right\vert )$ or $\mathrm{Sym}%
(L)$. The basic properties of $\mathrm{Sym}(L)$ are: (i) its identity
element is the identity permutation $\mathrm{id}=(1,2,...,L)$, (ii) it is
non-commutative for $L\geq 3$, and (iii) its cardinality is $L!$.

At this point, it is important to make a note on the composition of
permutations (and functions for that matter). There are two ways to define
the composition $\mathbf{r}\cdot \mathbf{s}$ of two permutations $\mathbf{r},%
\mathbf{s}\in \mathrm{Sym}(S)$. In the conventional definition (as a
\textquotedblleft left action\textquotedblright\ of the group on itself \cite%
{Amigo2025}), $(\mathbf{r}\cdot \mathbf{s})(k):=\mathbf{r}(\mathbf{s}%
(k))\equiv (\mathbf{r}\circ \mathbf{s})(k)\mathbf{\ }$for all $k\in S$,
where \textquotedblleft $\circ $\textquotedblright\ stands for the
composition of functions. In this case, the second permutation $\mathbf{s}$
acts first and the first permutation $\mathbf{r}$ acts second, so that%
\begin{equation}
\mathbf{r}\circ \mathbf{s}=(r_{1},r_{2},...,r_{L})\circ
(s_{1},s_{2},...,s_{L})=(r_{s_{1}},r_{s_{2}},...,r_{s_{L}}).
\label{rxs left}
\end{equation}%
However, it is more natural in calculations that\textbf{\ }$\mathbf{r}$ acts
first and then $\mathbf{s}$, in which case one speaks of a \textquotedblleft
right action\textquotedblright\ of the group on itself \cite{Amigo2025}.
Formally, one writes then $(k)(\mathbf{r}\cdot \mathbf{s}):=((k)\mathbf{r})%
\mathbf{s}\equiv (\mathbf{r}\ast \mathbf{s})(k)$ for all $k\in S$, i.e.,%
\begin{equation}
\mathbf{r}\ast \mathbf{s}=(r_{1},r_{2},...,r_{L})\ast
(s_{1},s_{2},...,s_{L})=(s_{r_{1}},s_{r_{2}},...,s_{r_{L}}).
\label{rxs right}
\end{equation}%
Note that 
\begin{equation}
\mathbf{r}\ast \mathbf{s}=\mathbf{s}\circ \mathbf{r}.  \label{left right}
\end{equation}

In this paper, we will use the definition (\ref{rxs right}) because this is
the usual choice in the literature, particularly in papers dealing with
transcripts. See Amig\'{o} and Dale \cite{Amigo2025} for a brief overview on
group actions where, however, the conventional composition (\ref{rxs left})
is chosen; use equation (\ref{left right}) to switch between both
definitions. For mathematical accounts, the interested reader is referred to
the books by Lang \cite{Lang2005} and Fraleigh \cite{Fraleigh2013}.

\begin{example}
For further references, the multiplication table of $(\mathrm{Sym}(3),\ast )$, the most common group in ordinal representations \cite{Bandt2019,Bandt2007}, is
\begin{equation}
\begin{tabular}{|c||l|l|l|l|l|l|}
\hline
$\mathbf{r}\ast \mathbf{s}$ & $123$ & $132$ & $213$ & $231$ & $312$ & $321$
\\ \hline\hline
$123$ & $123$ & $132$ & $213$ & $231$ & $312$ & $321$ \\ \hline
$132$ & $132$ & $123$ & $231$ & $213$ & $321$ & $312$ \\ \hline
$213$ & $213$ & $312$ & $123$ & $321$ & $132$ & $231$ \\ \hline
$231$ & $231$ & $321$ & $132$ & $312$ & $123$ & $213$ \\ \hline
$312$ & $312$ & $213$ & $321$ & $123$ & $231$ & $132$ \\ \hline
$321$ & $321$ & $231$ & $312$ & $132$ & $213$ & $123$ \\ \hline
\end{tabular}
\label{mult table Sym(3)B}
\end{equation}
where the composition $\mathbf{r}\ast \mathbf{s}$ is calculated as in equation
(\ref{rxs right}), $\mathbf{r}$ labels the rows (leftmost column) and $%
\mathbf{s}$ labels the columns (topmost row). The six permutations of $%
\mathrm{Sym}(3)$ are ordered lexicographically.
\end{example}

%%%%%%%%%%%%%%%%%%%%%%%%%%%%%%%%%%%%%%%%%%%%%%%%%%%%%%

\section{Algebraic representations}

\label{sec3}

In this section we generalize the ordinal representation of time series
discussed in Section \ref{sec2}. To this end, we consider symbolic
representations $\mathrm{\alpha }=(\alpha _{t})_{t\geq 0}$ of real-valued
time series $\mathrm{x}=(x_{t})_{t\geq 0}$ whose alphabet is now a general
group.

To begin with, discretization of a real-valued time series $\mathrm{x}%
=(x_{t})_{t\geq 0}$ produces a discrete-valued random sequence, both if $%
\mathrm{x}$ is a random trajectory and a (in general, projection of a higher
dimensional) deterministic orbit, as assumed in nonlinear time series
analysis.

To be more specific, let $\mathrm{x}=(x_{t})_{t\geq 0}$ be the orbit of the
initial condition $x_{0}$ generated by a measure-preserving dynamical system 
$(\Omega ,\mathcal{B},\mu ,f)$, i.e., (i) $(\Omega ,\mathcal{B},\mu )$ is a
probability space (where $\Omega $ is the state space of the system, $%
\mathcal{B}$ is a sigma-algebra of subsets of $\Omega $, and $\mu $ is a
positive measure on the measurable space $(\Omega ,\mathcal{B})$), (ii) $%
f:\Omega \rightarrow \Omega $ is a measurable mapping preserving the measure 
$\mu $ (that is, $\mu (f^{-1}(B))=\mu (B)$ for all $B\in \mathcal{B}$), and
(iii) $x_{t}=f^{t}(x_{0})$ (the $t$-th iterate of $x_{0}\in \Omega $ under $%
f $). Suppose that the state space $\Omega $ is discretized by the partition 
$\mathcal{P}=\{P_{1},P_{2},...,P_{\left\vert \mathcal{P}\right\vert }%
\mathcal{\}}$, where $P_{k}\in \mathcal{B}$ for $k=1,...,\left\vert \mathcal{%
P}\right\vert $. If $X_{t}^{\mathcal{P}}:\Omega \rightarrow
\{1,...,\left\vert \mathcal{P}\right\vert \}$ are the mappings defined as%
\begin{equation}
X_{t}^{\mathcal{P}}(x)=k\;\Leftrightarrow \;f^{t}(x)\in P_{k},
\label{symb dyn}
\end{equation}%
then $\mathbf{X}^{\mathcal{P}}:(X_{t}^{\mathcal{P}})_{t\geq 0}$ is a
stationary random process with alphabet (or state space) $\mathcal{A}%
=\{1,...,\left\vert \mathcal{P}\right\vert \}$, called the \textit{symbolic
dynamics of} $f$ \textit{with respect to} $\mathcal{P}$, whose probability
distribution is given by \cite{Hirata2023}%
\begin{eqnarray}
&&\left. \mathbb{P}(X_{t}^{\mathcal{P}}=k_{0},X_{t+1}^{\mathcal{P}%
}=k_{1},...,X_{t+n}^{\mathcal{P}}=k_{n})\right.  \label{P=mu} \\
&=&\mu (P_{k_{0}}\cap f^{-1}(P_{k_{1}})\cap ...\cap f^{-n}(P_{k_{n}})). 
\notag
\end{eqnarray}%
Therefore, the symbolic representation $\mathrm{\alpha }=(\alpha
_{t})_{t\geq 0}$ of $\mathrm{x}=(x_{t})_{t\geq 0}$, where $\alpha _{t}=k\in 
\mathcal{A}$ if $x_{t}=f^{t}(x_{0})\in P_{k}$, is, according to equation (%
\ref{symb dyn}), a realization of the random process $\mathbf{X}^{\mathcal{P}%
}$. See the books by Kitchens \cite{Kitchens1998} and Lind \cite{Lind2003}
for the general approach, and the book by Hao and Zhen \cite{Hao1998} for
applications of symbolic dynamics to Chaos Theory.

In particular, an ordinal representation of $\mathrm{x}$ is obtained by
coarse-graining $\Omega $ with the \textit{ordinal partition}

\begin{equation*}
\mathcal{P}_{L}=\{P_{\mathbf{r}}\neq \emptyset :\mathbf{r}\in \mathrm{Sym}(L)%
\mathcal{\}},
\end{equation*}%
where 
\begin{equation*}
P_{\mathbf{r}}=\{x\in \Omega :\mathrm{rank}(x,\,f(x),...,\,f^{L-1}(x))=%
\mathbf{r}\}.
\end{equation*}

In nonlinear time series analysis, the probability of any outcome is
estimated by their frequency (maximum likelihood estimator) in a
statistically significant sample of symbolic series $\mathbf{X}^{\mathcal{P}%
}(x_{0})=(k_{t})_{t\geq 0}$ (i.e., initial conditions $x_{0}\in \Omega $) or
a long enough time series. The resulting measure $\mu $, equation (\ref{P=mu}%
), is called \cite{Eckmann1985} the physical or natural measure of $\Omega $.

\begin{definition}
\label{DefAlgRep}A symbolic representation $\mathrm{\alpha }=(\alpha
_{t})_{t\geq 0}$ of a time series $\mathrm{x}=(x_{t})_{t\geq 0}$ is called
an \textit{algebraic representation} if the symbols $\alpha _{t}$ are
elements of a \textit{finite} group $(\mathcal{G},\cdot )$.
\end{definition}

In addition to the group $\mathrm{Sym}(L)$ of the ordinal representation
with embedding dimension $L$, examples of algebraic representations include $%
\mathcal{G=\{}0,1\mathcal{\}}$ endowed with the XOR operation (addition
modulo 2) and, more generally, $\mathcal{G=\{}0,1,...,n-1\mathcal{\}}$
endowed with addition modulo $n$. These groups arise in digital
communications and cryptography. Henceforth, all groups are finite.

\begin{remark}
\label{RemarkAlgRep}

\begin{enumerate}
\item Among the different algebraic symbols for the rank relationships (\ref%
{ord patt}) and (\ref{ord patt 2}) mentioned in Remark \ref%
{RemarkAlternativeSymbols}, only the $0$-$1$ square matrices proposed by
Bunk et al. \cite{Bunk2013} build a group (with respect to matrix
multiplication). More generally, for a subset $S$ of a group $\mathcal{G}$
to be the alphabet of an algebraic representation, $S$ must be a subgroup of 
$\mathcal{G}$.

\item A different situation occurs when the alphabet of a symbolic
representation can be endowed with an algebraic structure. For example, the
alphabet $\mathcal{A}=\{1,...,\left\vert \mathcal{P}\right\vert \}$ of the
symbolic dynamics with respect to a partition $\mathcal{P}$, equation (\ref%
{symb dyn}), can be transformed into a group just by adding the letters
modulo $\left\vert \mathcal{P}\right\vert $.
\end{enumerate}
\end{remark}

It is remarkable that, according to the following theorem of Group Theory,
called the \textit{Cayley theorem} \cite{Herstein1996,Amigo2025},
permutations can be seen as universal symbols for algebraic representations.

\begin{theorem}
\label{ThmCayley}(Cayley) Every group $(\mathcal{G},\cdot )$ is isomorphic
to a subgroup of $(\mathrm{Sym}(\left\vert \mathcal{G}\right\vert ,\ast )$,
i.e., there is a mapping (called a group homomorphism) $\Phi :\mathcal{G}%
\rightarrow \mathrm{Sym}(\left\vert \mathcal{G}\right\vert )$ such that

\begin{description}
\item[(P1)] $\Phi $ sends the identity element $e$ of $\mathcal{G}$ to the
identity permutation $(1,2,...,\left\vert \mathcal{G}\right\vert )$ of $%
\mathrm{Sym}(\left\vert \mathcal{G}\right\vert \mathcal{)}$;

\item[(P2)] $\Phi (a\cdot b)$ $=\Phi (a)\ast \Phi (b)$ for all $a,b\in 
\mathcal{G}$. Hence, $\Phi (a^{-1})=(\Phi (a))^{-1}$.
\end{description}
\end{theorem}

Therefore, if we restrict the range of the homomorphism $\Phi $ to its
image, $\Phi (\mathcal{G})$, we obtain a group isomorphism $\mathcal{G}%
\rightarrow \Phi (\mathcal{G})$, i.e., a bijective (one-to-one, invertible)
homomorphism, called Cayley's isomorphism which, abusing notation, will also
be denoted by $\Phi $. In this case, $\mathcal{H}:=\Phi (\mathcal{G})$ is a
subgroup of $\mathrm{Sym}(\left\vert \mathcal{G}\right\vert )$ with
cardinality $\left\vert \mathcal{H}\right\vert =\left\vert \mathcal{G}%
\right\vert $. Sometimes we also call the homomorphism $\Phi :\mathcal{G}%
\rightarrow \mathrm{Sym}(\left\vert \mathcal{G}\right\vert )$ an embedding.

In view of Cayley's theorem, any group $\mathcal{G}$ can be represented by
means of $\left\vert \mathcal{G}\right\vert $ permutations on the set $%
\{1,2,...,\left\vert \mathcal{G}\right\vert !\}$. From a practical point of
view, finding the minimal-order symmetric group into which a given group $%
\mathcal{G}$ embeds in $\mathrm{Sym}(\mathcal{G})$ is rather difficult \cite%
{Johnson1971,Grechkoseeva2003}. On the other hand, Cayley's theorem is
practical when it comes to endow groups of low and moderate cardinality with
a metric, as we will discuss in Section \ref{sec62}.

Cayley's isomorphism can be implemented in several ways. In Section \ref%
{sec4} we are going to learn some of them.

%%%%%%%%%%%%%%%%%%%%%%%%%%%%%%%%%%%%%%%%%%%%%%%

\section{Transcripts}

\label{sec4}

Ordinal representations are very popular in time series analysis \cite%
{Zanin2012,Amigo2015B}. However, very few applications have harnessed the
algebraic structure of the symmetric group so far. In fact, the only
applications we are aware of build on the concept of transcript, introduced
by Monetti et al. \cite{Monetti2009}. In this section, we consider this
concept in the more general framework of algebraic representations and
discuss its main properties, before revisiting some applications in Section %
\ref{sec5}.

\begin{definition}
\label{DefTranscriptionMap}Given a finite group $(\mathcal{G},\cdot )$, the 
\emph{transcription mapping} $T:\mathcal{G}\times \mathcal{G}\rightarrow 
\mathcal{G}$ is defined as%
\begin{equation}
T(a,b)=b\cdot a^{-1}.  \label{T(a,b)}
\end{equation}%
The group element $T(a,b)$ is called the transcript from the
source $a$ to the target $b$.
\end{definition}

The basic properties of the transcription mapping are the following.

\begin{description}
\item[(T1)] $T(b,a)=T(a,b)^{-1}$.

\item[(T2)] $T(b,c)\cdot T(a,b)=T(a,c).$

\item[(T3)] $T$ is a $\left\vert \mathcal{G}\right\vert $-to-1 mapping. In
fact, for $\forall c\in \mathcal{G}$, 
\begin{equation*}
T^{-1}(c)=\{(a,\,c\cdot a),\;\forall a\in \mathcal{G}\}=\{(c^{-1}\cdot
b,\,b),\;\forall b\in \mathcal{G}\}.
\end{equation*}%
Trivially, the sets $\{T^{-1}(c):c\in \mathcal{G}\}$ build a partition of $%
\mathcal{G}\times \mathcal{G}$.

\item[(T4)] \textit{Equivalence property}. Given a triple $(a,b,T(a,b))$,
any pair of elements (i.e., $(a,b)$, $(a,T(a,b))$ or $(b,T(a,b))$)
univocally determines the remaining element. This property is instrumental
in many proofs involving transcripts.
\end{description}

\begin{example}
For $\mathcal{G}=\mathrm{Sym}(3)$, we obtain the following transcripts $T(%
\mathbf{r},\mathbf{s})=\mathbf{s}\ast \mathbf{r}^{-1}$,%
\begin{equation*}
\begin{tabular}{|c||l|l|l|l|l|l|}
\hline
$\mathbf{s}\ast \mathbf{r}^{-1}$ & $123$ & $132$ & $213$ & $231$ & $312$ & $%
321$ \\ \hline\hline
$123$ & $123$ & $132$ & $213$ & $231$ & $312$ & $321$ \\ \hline
$132$ & $132$ & $123$ & $312$ & $321$ & $213$ & $231$ \\ \hline
$213$ & $213$ & $231$ & $123$ & $132$ & $321$ & $312$ \\ \hline
$231$ & $312$ & $321$ & $132$ & $123$ & $231$ & $213$ \\ \hline
$312$ & $231$ & $213$ & $321$ & $312$ & $123$ & $132$ \\ \hline
$321$ & $321$ & $312$ & $231$ & $213$ & $132$ & $123$ \\ \hline
\end{tabular}%
\end{equation*}%
where the source permutation $\mathbf{r}$ labels the rows and the target
permutation $\mathbf{s}$ labels the columns.
\end{example}

If $(\mathcal{G},\cdot )$ is non-abelian, a \emph{conjugate transcription
mapping} $\tilde{T}:\mathcal{G}\times \mathcal{G}\rightarrow \mathcal{G}$
can be defined as%
\begin{equation}
\tilde{T}(a,b)=a^{-1}\cdot b.  \label{T(a,b) conjugate}
\end{equation}%
Note that%
\begin{equation}
\tilde{T}(a,b)=T(b^{-1},a^{-1}).  \label{T conj = T}
\end{equation}%
The transformation $\tilde{T}_{a}:=\tilde{T}(a,\cdot ):\mathcal{G}%
\rightarrow \mathcal{G}$, i.e.,%
\begin{equation}
\tilde{T}_{a}(b):=\tilde{T}(a,b)=a^{-1}\cdot b  \label{Ta(b) conjugate}
\end{equation}%
is called a \textit{left} \textit{translation by} $a$.

The following result shows how to implement Cayley's isomorphism via
conjugate transcripts.

\begin{theorem}
\label{ThmCayleyViaTranscripts}For each $a\in \mathcal{G}$, the mapping $%
\Phi :a\mapsto \tilde{T}_{a}=\tilde{T}(a,\cdot )$ is an isomorphism from $(%
\mathcal{G},\cdot )$ onto a subgroup of $(\mathrm{Sym}(\mathcal{G}),\ast )$.
\end{theorem}

\begin{proof}
First of all, we show that $\tilde{T}_{a}\in \mathrm{Sym}(\mathcal{G})$,
i.e., it is a bijection from $\mathcal{G}$ to $\mathcal{G}$, for all $a\in 
\mathcal{G}$. This is true because every $b\in \mathcal{G}$ has an inverse
(or anti-image) under $\tilde{T}_{a}$, namely, $(\tilde{T}_{a}(b))^{-1}=%
\tilde{T}_{a^{-1}}(b)$. Indeed,%
\begin{equation*}
(\tilde{T}_{a}\ast \tilde{T}_{a^{-1}})(b)=\tilde{T}_{a^{-1}}(\tilde{T}%
_{a}(b))=a\cdot (a^{-1}\cdot b)=(a\cdot a^{-1})\cdot b=b
\end{equation*}%
for all $b\in \mathcal{G}$. Analogously, $\tilde{T}_{a^{-1}}\ast $ $\tilde{T}%
_{a}=\mathrm{id}.$

To show that $\Phi :a\mapsto \tilde{T}_{a}$ is a group homomorphism from $%
\mathcal{G}$ into $\mathrm{Sym}(\mathcal{G})$, we have to prove that
Properties (P1) and (P2) in Theorem \ref{ThmCayley} hold true.

(P1) Let $e$ be the identity element of $\mathcal{G}$. Then, $\tilde{T}%
_{e}(a)=e^{-1}\cdot a=a$ for all $a\in \mathcal{G}$, which shows that $%
\tilde{T}_{e}:\mathcal{G}\rightarrow \mathcal{G}$ is the identity mapping.

(P2) Furthermore,%
\begin{eqnarray*}
\tilde{T}_{a\cdot b}(c) &=&(a\cdot b)^{-1}\cdot c=(b^{-1}\cdot a^{-1})\cdot
c=b^{-1}\cdot (a^{-1}\cdot c) \\
&=&\tilde{T}_{b}(\tilde{T}_{a}(c))=(\tilde{T}_{a}\ast \tilde{T}_{b})(c)
\end{eqnarray*}%
for all $a,b,c\in \mathcal{G}$.

Hence, $\mathcal{G}$ is isomorphic to its image $\Phi (\mathcal{G})$ under $%
\Phi $.
\end{proof}

\bigskip

The implementation $\Phi :a\mapsto \tilde{T}_{a}$ of Cayley's isomorphism,
Theorem \ref{ThmCayleyViaTranscripts}, can be visualized as follows. Let $%
\{a_{1},a_{2},...,a_{n}\}$ be any ordering of the elements of $\mathcal{G}$,
so that the conjugate transcription mapping $\tilde{T}:\mathcal{G}\times 
\mathcal{G}\rightarrow \mathcal{G}$ can be represented by the $n\times n$
matrix $(\tilde{T}(a_{i},a_{j}))_{1\leq i,j\leq n}$, shortened as $(\tilde{T}%
(i,j))_{1\leq i,j\leq n}=(\tilde{T}_{i}(j))_{1\leq i,j\leq n}$ in the
following table: 
\begin{equation}
\begin{tabular}{|c||c|c|c|c|c|}
\hline
& $1$ & $\cdots $ & $j$ & $\cdots $ & $n$ \\ \hline\hline
$1$ & $\tilde{T}_{1}(1)$ & $\cdots $ & $\tilde{T}_{1}(j)$ & $\cdots $ & $%
\tilde{T}_{1}(n)$ \\ \hline
$\vdots $ & $\vdots $ & $\vdots $ & $\vdots $ & $\vdots $ & $\vdots $ \\ 
\hline
$i$ & $\tilde{T}_{i}(1)$ & $\cdots $ & $\tilde{T}_{i}(j)$ & $\cdots $ & $%
\tilde{T}_{i}(n)$ \\ \hline
$\vdots $ & $\vdots $ & $\vdots $ & $\vdots $ & $\vdots $ & $\vdots $ \\ 
\hline
$n$ & $\tilde{T}_{n}(1)$ & $\cdots $ & $\tilde{T}_{n}(j)$ & $\cdots $ & $%
\tilde{T}_{n}(n)$ \\ \hline
\end{tabular}
\label{MatrixTranscript}
\end{equation}%
Therefore, the $i$th row of the conjugate transcription matrix (\ref%
{MatrixTranscript}) is precisely the one-line form of the permutation $%
\tilde{T}_{i}:=\tilde{T}(i,\cdot )$, $1\leq i\leq n$:%
\begin{equation}
\Phi (a_{i})=\tilde{T}_{a_{i}}\equiv \tilde{T}_{i}=\left( 
\begin{array}{ccccc}
1 & \cdots & j & \cdots & n \\ 
\tilde{T}_{i}(1) & \cdots & \tilde{T}_{i}(j) & \cdots & \tilde{T}_{i}(n)%
\end{array}%
\right) .  \label{Phi(a)}
\end{equation}

\begin{example}
\label{ExampleKlein}
Let us particularize the procedure (\ref{MatrixTranscript})-(\ref{Phi(a)})
to the Klein four-group  $(\mathcal{K},\cdot )$, defined by the
multiplication table%
\begin{equation}
\begin{tabular}{|c||c|c|c|c|}
\hline
$\cdot $ & $e$ & $a$ & $b$ & $c$ \\ \hline\hline
$e$ & $e$ & $a$ & $b$ & $c$ \\ \hline
$a$ & $a$ & $e$ & $c$ & $b$ \\ \hline
$b$ & $b$ & $c$ & $e$ & $a$ \\ \hline
$c$ & $c$ & $b$ & $a$ & $e$ \\ \hline
\end{tabular}
\label{TableK}
\end{equation}
Note that $\mathcal{K}$ is abelian (as any group whose cardinality is the
square of a prime number) since the multiplication table in equation (\ref%
{TableK}) is symmetric and every element other than the identity $e$ has
order $2$, i.e., every element is its own inverse. Therefore, $\tilde{T}%
_{r}(s)=r^{-1}\cdot s=r\cdot s$ for all $r,s\in \mathcal{K}$, so that the
table for the left translations $\tilde{T}_{r}\in \mathrm{Sym}(\mathcal{K})=%
\mathrm{Sym}(4)$ is obtained by just copying the multiplication table (\ref%
{TableK}): 
\begin{equation}
\begin{tabular}{|c||c|c|c|c|}
\hline
& $e$ & $a$ & $b$ & $c$ \\ \hline\hline
$\tilde{T}_{e}$ & $e$ & $a$ & $b$ & $c$ \\ \hline
$\tilde{T}_{a}$ & $a$ & $e$ & $c$ & $b$ \\ \hline
$\tilde{T}_{b}$ & $b$ & $c$ & $e$ & $a$ \\ \hline
$\tilde{T}_{c}$ & $c$ & $b$ & $a$ & $e$ \\ \hline
\end{tabular}
\label{TableK2}
\end{equation}%
Thus, the isomorphic copy $\tilde{T}_{b}$ of $b \in \mathcal{K}$ is the permutation given by the third row of the table
in equation (\ref{TableK}): 
\begin{equation}
\tilde{T}_{b}=\left( 
\begin{tabular}{llll}
$e$ & $a$ & $b$ & $c$ \\ 
$b$ & $c$ & $e$ & $a$%
\end{tabular}%
\right) 
= bcea.  \label{T_b}
\end{equation}%
If the elements of $\mathcal{K}$ are encoded as, say,%
\begin{equation}
e=1,\,a=2,\,b=3,\,c=4,  \label{encoding}
\end{equation}%
%then the permutation $\tilde{T}_{b}:\mathcal{K}\rightarrow \mathcal{K}$ in
%equation (\ref{T_b}) reads%
%\begin{equation}
%  \tilde{T}_{b}\equiv \tilde{T}_{3}=3412  
%\end{equation}
%in the one-line form.
then
\begin{equation}
  \tilde{T}_{1}=1234, \; \tilde{T}_{2}=2143, \;
  \tilde{T}_{3}=3412, \; \tilde{T}_{4}=4321.
  \label{encoding Klein2}
\end{equation}
\end{example}

The concept of transcription translates to time series in a straightforward
way. Let $\mathrm{x}=(x_{t})_{t\geq 0}$ and $\mathrm{y}=(y_{t})_{t\geq 0}$
be two time series and let $\mathrm{\alpha }=(\alpha _{t})_{t\geq 0}$ and $%
\mathrm{\beta }=(\beta _{t})_{t\geq 0}$ be algebraic representations of $%
\mathrm{x}$ and $\mathrm{y}$ over the same group $\mathcal{G}$. Remember
that $\mathrm{\alpha }$ and $\mathrm{\beta }$ are $\mathcal{G}$-valued
random trajectories (Section \ref{sec3}).

\begin{definition}
\label{DefTranscriptTS}Let $T:\mathcal{G}\times \mathcal{G}\rightarrow 
\mathcal{G}$ be the transcription mapping defined in Definition \ref%
{DefTranscriptionMap}. The transcription with source $\mathrm{\alpha }$,
target $\mathrm{\beta }$, and coupling delay $\Lambda \in \mathbb{Z}$ is the 
$\mathcal{G}$-valued time series%
\begin{equation}
\mathrm{T}_{\mathrm{\alpha },\mathrm{\beta }_{\Lambda }}=(T(\alpha
_{t},\beta _{t+\Lambda }))_{t\geq t_{0}}=(\beta _{t+\Lambda }\ast \alpha
_{t}^{-1})_{t\geq t_{0}}  \label{tau(Lambda)}
\end{equation}%
where $t_{0}=\max \{0,-\Lambda \}$.
\end{definition}

Without loss of generality, we apply the coupling delay to the target (see property
(T1) above). If the source and target are known from the context, we will
shorten $\mathrm{T}_{\mathrm{\alpha },\mathrm{\beta }_{\Lambda }}=(\beta
_{t+\Lambda }\ast \alpha _{t}^{-1})_{t\geq t_{0}}$ as%
\begin{equation}
\mathrm{T}_{\mathrm{\alpha },\mathrm{\beta }_{\Lambda }}=(\tau _{t}(\Lambda
))_{t\geq t_{0}},  \label{tau(Lambda)2}
\end{equation}%
with 
\begin{equation}
\tau _{t}(\Lambda ):=\beta _{t+\Lambda }\ast \alpha _{t}^{-1}\in \mathcal{G}.
\label{tau(Lambda)3}
\end{equation}%
If $\Lambda =0$, then we set $\tau _{t}(0)=\tau _{t}=\beta _{t}\ast \alpha
_{t}^{-1}$.

The fact that $t_{0}=\left\vert \Lambda \right\vert >0$ for $\Lambda <0$ is
irrelevant in practice. Alternatively, in that case one could set $\beta
_{-\left\vert \Lambda \right\vert }=\beta _{-\left\vert \Lambda \right\vert
+1}=\ldots =\beta _{-1}:=e$ and start at $t=0$ in Equation (\ref{tau(Lambda)}%
), as we will do from now on for notational simplicity. If $\mathrm{y}=%
\mathrm{x}$ and $\Lambda \neq 0$, then we call $\tau _{t}(\Lambda )=\alpha
_{t+\Lambda }\ast \alpha _{t}^{-1}$, $t\geq 0$, \textit{self-transcripts};
otherwise, if $\mathrm{y}\neq \mathrm{x}$ and $\Lambda \in \mathbb{Z}$, we
speak of \textit{cross-transcripts} or simply transcripts. As its name
indicates, the coupling delay $\Lambda $ allows us to study the dependencies
between different variables of a time series (self-transcripts) or two time
series (cross-transcripts).

Definition \ref{DefTranscriptTS} can be generalized to more than two
algebraic representations over of the same group. Thus, if $\mathrm{\alpha }%
^{m}=(\alpha _{t}^{m})_{t\geq 0}$, $m=1,2,...,M\geq 2$, are $\mathcal{G}$%
-valued representations, then we can define a (self- or cross-)
transcription for any of the $M^{2}$ possible pairings $\mathrm{\alpha }^{j},%
\mathrm{\alpha }^{k}$ of the algebraic representations $\mathrm{\alpha }%
^{1},...,\mathrm{\alpha }^{M}$:%
\begin{equation}
\mathrm{T}_{\mathrm{\alpha }^{j},\mathrm{\alpha }_{\Lambda
(j,k)}^{k}}=(T(\alpha _{t}^{j},\alpha _{t+\Lambda (j,k)}^{k}))_{t\geq 0},
\label{M transcripts}
\end{equation}%
where we allow for coupling delays that depend on the source and the target.
Abusing notation, we will use the notations $\mathrm{\alpha }^{j}$, $\mathrm{%
\alpha }_{\Lambda (j,k)}^{k}$, $T_{\mathrm{\alpha }^{j},\mathrm{\alpha }%
_{\Lambda (j,k)}^{k}}$ and the like for both time series and the $\mathcal{G}
$-valued random processes that generate them.

%%%%%%%%%%%%%%%%%%%%%%%%%%%%%%%%%%%%%%%%%%%%%%%%%%%%5

\section{Transcript-based tools and applications}

\label{sec5}

In this section we review several transcript-based tools in use in time series analysis. The traditional tools include the entropic (Sections \ref{sec51}-\ref{sec54}) and algebraic (\ref{sec53}) ones. The most recent tools are the Cayley and Kendall distances between permutations (\ref{sec61}) and their extensions to general groups (Section \ref{sec62}), which also fall into the category of algebraic tools. The noise robustness of transcript-based tools has been studied by Adams and Lehnertz \cite{Adams2025}. For the information-theoretic concepts below, see the book by Cover and Thomas \cite{Cover2006}.

%%%%%%%%%%%%%%%%%%%%%%%%%%%%%%%%%%%%%%%%%%%%%%%%%%%

\subsection{Entropy}

\label{sec51}

One of the main tools in nonlinear time series analysis is entropy. Let $%
\mathrm{\alpha }=(\alpha _{t})_{t\geq 0}$, $\mathrm{\beta }=(\beta
_{t})_{t\geq 0}$ be two stationary $\mathcal{G}$-valued random processes,
and $T_{\mathrm{\alpha },\mathrm{\beta }_{\Lambda }}=(T_{\alpha _{t},\beta
_{t+\Lambda }})_{t\geq 0}$ the random process defined in equations (\ref%
{tau(Lambda)})-(\ref{tau(Lambda)3})). The \textit{transcript entropy} of the
source process $\mathrm{\alpha }$ and target process $\mathrm{\beta }$ with
coupling delay $\Lambda \in \mathbb{Z}$ is the Shannon entropy of the random
variable $T_{\alpha _{t},\beta _{t+\Lambda }}=\beta _{t+\Lambda }\ast \alpha
_{t}^{-1}$ with probability distribution $P_{\tau }(\Lambda )=\{p(\tau
_{t}(\Lambda )):\tau _{t}(\Lambda )\in \mathcal{G}\}$, i.e., 
\begin{equation}
H(T_{\alpha _{t},\beta _{t+\Lambda }})=H(P_{\tau }(\Lambda ))=-\sum_{\tau
_{t}(\Lambda )\in \mathcal{G}}p(\tau _{t}(\Lambda ))\log p(\tau _{t}(\Lambda
)),  \label{H(P)}
\end{equation}%
where $t$ is fixed but otherwise arbitrary due to stationarity,
\begin{equation}
p(\tau _{t}(\Lambda ))=\sum_{\beta _{t+\Lambda }\ast \alpha _{t}^{-1}=\tau
_{t}(\Lambda )}p(\alpha _{t},\beta _{t+\Lambda })=\sum_{\alpha _{t}\in 
\mathcal{G}}p(\alpha _{t},\tau _{t}(\Lambda )\ast \alpha _{t}),
\label{p(tau)}
\end{equation}%
$p(\alpha _{t},\beta _{t+\Lambda })$ being the joint probability of the
symbols $\alpha _{t}$ and $\beta _{t+\Lambda }$. The usual logarithm bases
are $e$ and $2$.

In the particular case $\mathrm{\beta }=\mathrm{\alpha }$ and $\Lambda \neq
0 $, we speak of the self-transcript entropy $H(T_{\alpha
_{t},\alpha _{t+\Lambda }})$ of the random process $\mathrm{\alpha }$ with
coupling delay $\Lambda $; otherwise we speak of cross-transcript entropy or
simply transcript entropy. By property (T1) of Section \ref{sec4}, $%
H(T_{\alpha _{t},\beta _{t+\Lambda }})=H(T_{\beta _{t+\Lambda },\alpha
_{t}}) $. On the other hand, $H(T_{\alpha _{t},\beta _{t+\Lambda }})\neq
H(T_{\alpha _{t+\Lambda },\beta _{t}})$ in general, unless $\Lambda =0$.

Since $H(T_{\alpha _{t},\beta _{t+\Lambda }})$ ranges between $0$ and $\log \left\vert \mathcal{G}\right\vert $,
the normalized entropy%
\begin{equation}
h(T_{\alpha _{t},\beta _{t+\Lambda }}):=\frac{H(T_{\alpha _{t},\beta
_{t+\Lambda }})}{\log \left\vert \mathcal{G}\right\vert }=\frac{H(P_{\tau
}(\Lambda ))}{\log \left\vert \mathcal{G}\right\vert }=:h(P_{\tau }(\Lambda ))
\label{normalized H}
\end{equation}%
is preferred in applications. If $\Lambda =0$, we write $P_{\tau
}(0)=P_{\tau }=\{p(\tau _{t}):\tau _{t}\in \mathcal{G}\}$.

More generally, the Shannon entropy of a probability distribution of ordinal
patterns is called permutation entropy \cite{Bandt2002,Amigo2013}. Of
course, the Shannon entropy can be replaced by any other entropy in the
equation (\ref{H(P)}), in particular by the R\'{e}nyi or Tsallis entropies 
\cite{Renyi1961,Tsallis1988}, which are one-parameter families of
generalized entropies that include the Shannon entropy.

In the case of drive-response systems, the driver may have forbidden
symbols, as actually happens in the ordinal representations. However, this
does not mean that there will also be forbidden transcripts because, as
shown in Section \ref{sec3}, there are $\left\vert \mathcal{G}\right\vert $
pairs $(\alpha _{t},\beta _{t+\Lambda })$ such that $\beta _{t+\Lambda }\ast
\alpha _{t}^{-1}=\tau _{t}(\Lambda ).$ So, we expect that, in general, $%
p(\tau _{t}(\Lambda ))>0$ for all $\tau _{t}(\Lambda )\in \mathcal{G}$.
There are exceptions, though: full and generalized synchronization. We come
back to this point in Section \ref{sec72}.

Transcript entropy has been used to study the brain activity of a subject
suffering from frontal lobe epilepsy \cite{Amigo2012}.

Differentiation of patients with obstructive sleep apnea from healthy
controls has been achieved based on the heart rate--blood pressure coupling
quantified by entropic indices, including self- and cross-transcript
entropies \cite{Pilarczyk2023}. Remarkably, all 10 best triplets of
classifiers include the self-transcript entropy of the heart-rate time series with different patterns lengths and coupling delays.

Entropy is instrumental for calculating divergences, statistical
complexity, coupling complexity coefficients, and mutual information in the
forthcoming Sections \ref{sec52}-\ref{sec54}. For the estimation of entropy
in time series analysis, see Paninski \cite{Paninski2003}. For practical
considerations of permutation entropy, see Riedl et al. \cite{Riedl2013}.
The statistical properties of permutation entropy are studied in Chagas et
al. \cite{Chagas2022}.

%%%%%%%%%%%%%%%%%%%%%%%%%%%%%%%%%%%%%%%%%%%

\subsection{Divergence and statistical complexity}

\label{sec52}

Again, let $P_{\mathbf{\tau }}=\{p(\tau _{t}):\tau _{t}\in \mathcal{G}\}$ be
a (theoretical or empirical) probability distribution of transcripts
obtained from two stationary $\mathcal{G}$-valued random processes $\mathrm{%
\alpha }=(\alpha _{t})_{t\geq 0}$, $\mathrm{\beta }=(\beta _{t})_{t\geq 0}$
as in equation (\ref{p(tau)}), and $Q^{\mathrm{ref}}=\{q(\tau _{t}):\tau
_{t}\in \mathcal{G}\}$ be a \textquotedblleft reference\textquotedblright\
probability distribution. In case of a coupling delay $\Lambda \neq 0$,
replace $\tau _{t}$ by $\tau _{t}(\Lambda )$.

The \textit{Kullback-Leibler} \textit{(KL)} \textit{divergence} (or relative
entropy) of $P_{\mathbf{\tau }}$ and $Q^{\mathrm{ref}}$, a type of
statistical distance or a measure of similarity between probability
distributions, is defined as%
\begin{equation}
D_{KL}(P_{\mathbf{\tau }}\parallel Q^{\mathrm{ref}})=\sum_{\tau _{t}\in 
\mathcal{G}}p(\tau _{t})\log \frac{p(\tau _{t})}{q(\tau _{t})}.  \label{D_KL}
\end{equation}%
Note that $D_{KL}(P_{\mathbf{\tau }}\parallel Q^{\mathrm{ref}})$ is in
general not symmetric with respect to the exchange of $P_{\mathbf{\tau }}$
and $Q^{\mathrm{ref}}$. Therefore, a symmetrized form $D_{KL}^{\mathrm{sym}%
}(P_{\mathbf{\tau }}\parallel Q^{\mathrm{ref}})$, such as the arithmetic or
harmonic mean of $D_{KL}(P_{\mathbf{\tau }}\parallel Q^{\mathrm{ref}})$ and $%
D_{KL}(Q^{\mathrm{ref}}\parallel P_{\mathbf{\tau }})$, is preferred in practice.

But perhaps the most popular symmetric divergence currently is the \textit{%
Jensen-Shannon (JS) divergence}, which, particularized to the probabilities $%
P_{\mathbf{\tau }}$ and $Q^{\mathrm{ref}}$, is defined as 
\begin{equation}
D_{JS}(P_{\mathbf{\tau }}\parallel Q^{\mathrm{ref}})=H\left( \frac{P_{%
\mathbf{\tau }}+Q^{\mathrm{ref}}}{2}\right) -\frac{1}{2}\left( H(P_{\mathbf{%
\tau }})+H(Q^{\mathrm{ref}})\right) ,  \label{D_JS}
\end{equation}%
where $H(\cdot )$ is \ the Shannon entropy. Among its properties we
highlight: (i) $0\leq D_{JS}(P\parallel Q)\leq \log 2$ for any two
probability distributions $P$ and $Q$ on the same finite state space (so
that the JS divergence ranges in the interval $[0,1]$ if base $2$ logarithms
are used \cite{Lin1991}), (ii) the maximum value is achieved when $P$ and $Q$
have disjoint supports (meaning there is no event where both $P$ and $Q$
assign a non-zero probability), and (iii) $\left( D_{JS}(P\parallel
Q)\right) ^{1/2}$ is formally a distance for probability distributions,
called the Jensen-Shannon distance \cite{Endres2003}.

Several choices are common for $Q^{\mathrm{ref}}$. The simplest one is the
uniform distribution $U$ on $\mathcal{G}$. Another choice is $P_{\tau }^{%
\mathrm{ind}}=\{p_{\mathrm{ind}}(\tau _{t}):\tau _{t}\in \mathcal{G}\}$,
where \cite{Monetti2009} 
\begin{equation}
p_{\mathrm{ind}}(\tau _{t})=\sum_{\beta _{t}\ast \alpha _{t}^{-1}=\tau
_{t}}p(\alpha _{t})p(\beta _{t})  \label{Pind}
\end{equation}%
and $t$ is fixed but otherwise arbitrary due to the assumed stationarity of the time series. In other words, if $P_{\mathrm{\alpha }}=\{p(\alpha _{t}):\alpha _{t}\in 
\mathcal{G}\}$ and $P_{\mathrm{\beta }}=\{p(\beta _{t}):\beta _{t}\in 
\mathcal{G}\}$, then $P_{\tau }^{\mathrm{ind}}$ is the probability
distribution obtained from the product distribution $P_{\mathrm{\alpha }%
}\times P_{\mathrm{\beta }}$ by adding the entries $p(\alpha _{t})p(\beta
_{t})$ such that $\beta _{t}\ast \alpha _{t}^{-1}=\tau _{t}$ for each $\tau
_{t}\in \mathcal{G}$. Therefore, $D_{JS}(P_{\mathbf{\tau }}\parallel P_{\tau
}^{\mathrm{ind}})$ and $D_{JS}(P_{\mathbf{\tau }}\parallel
P_{\tau }^{\mathrm{ind}})^{1/2}$ can be considered nonlinear measures of
statistical dependence between the processes $\mathrm{\alpha }$ and $\mathrm{%
\beta }$.

The product of the JS divergence $D_{JS}(P\parallel U)$ (where $P$ stands
for any probability distribution) times the Shannon entropy $H(P)$ was
introduced by Rosso et al. \cite{Rosso2007} and called the \textit{%
statistical complexity }of the probability distribution $P$,%
\begin{equation}
\mathrm{SC}(P)=D_{JS}(P\parallel U)H(P).  \label{SC_JensenShannon}
\end{equation}%
Usually, definition (\ref{SC_JensenShannon}) is provided with an additional
normalization factor. Of course, if base $2$ logarithms are used (so that $%
0\leq D_{JS}(P\parallel U)\leq 1$) and $H(P)$ is replaced with the
normalized entropy $h(P)=H(P)/\log _{2}\left\vert \mathcal{G}\right\vert $
(equation (\ref{normalized H})), then $0\leq \mathrm{SC}(P)\leq 1$
automatically. By definition, $\mathrm{SC}(P)=0$ when $P$ is uniform ($P=U$)
or when $P$ is deterministic, reaching its maximum between those extreme
situations.

Analogously, we call%
\begin{equation}
\mathrm{SC}(P_{\tau},Q^{\mathrm{ref}})=D^{\mathrm{sym}}(P_{\tau}\parallel Q^{\mathrm{ref}%
})H(P_{\tau})  \label{SC(P)}
\end{equation}%
(possibly with a normalization factor) the \textit{transcript statistical
complexity} of the transcript probability distribution $P_{\tau}$ with respect to $Q^{\mathrm{%
ref}}$, where $D^{\mathrm{sym}}(P_{\mathbf{\tau }}\parallel Q^{\mathrm{ref}%
}) $ is a divergence symmetric in $P_{\tau}$ and $Q^{\mathrm{ref}}$.

Usually, $D^{\mathrm{sym}}(P_{\mathbf{\tau }}\parallel Q^{\mathrm{ref}})$ is
the JS divergence but other choices have been used in the literature, too.
Thus, the harmonic mean%
\begin{equation*}
D_{KL}^{\mathrm{sym}}(P_{\mathbf{\tau }}\parallel P_{\tau }^{\mathrm{ind}})=%
\frac{D_{KL}(P_{\mathbf{\tau }}\parallel P_{\tau }^{\mathrm{ind}%
})D_{KL}(P_{\tau }^{\mathrm{ind}}\parallel P_{\mathbf{\tau }})}{D_{KL}(P_{%
\mathbf{\tau }}\parallel P_{\tau }^{\mathrm{ind}})+D_{KL}(P_{\tau }^{\mathrm{%
ind}}\parallel P_{\mathbf{\tau }})}
\end{equation*}%
has been used \cite{Monetti2009} to characterize the synchronization regimes
of a bidirectionally coupled R\"{o}ssler-R\"{o}ssler system.

Likewise, $\mathrm{SC}_{KL}(P_{\mathbf{\tau }},P_{\tau }^{\mathrm{ind}%
})=D_{KL}^{\mathrm{sym}}(P_{\mathbf{\tau }}\parallel P_{\tau }^{\mathrm{ind}%
})H(P_{\tau })$ has been used \cite{Amigo2012} to measure the complexity of
two bidirectionally coupled R\"{o}ssler oscillators and two delay-coupled
logistic maps.

%%%%%%%%%%%%%%%%%%%%%%%%%%%%%%%%%%%%%%%%%%%%%%%%%%

\subsection{Mutual information and coupling complexity coefficients}

\label{sec54}

Given $M\geq 2$ stationary $\mathcal{G}$-valued random processes $\mathrm{%
\alpha }^{m}=(\alpha _{t}^{m})_{t\geq 0}$, $1\leq m\leq M$, their \textit{coupling
complexity coefficient} (CCC) is defined as \cite{Monetti2013,Monetti2013B}%
\begin{equation}
C(\mathrm{\alpha }^{1},...,\mathrm{\alpha }^{M})=\min_{1\leq m\leq
M}I(\alpha _{t}^{m};T_{\alpha _{t}^{1},\alpha _{t}^{2}},...,T_{\alpha
_{t}^{M-1},\alpha _{t}^{M}}),  \label{CCC}
\end{equation}%
where $T_{\alpha _{t}^{m},\alpha _{t}^{m+1}}$ is the random variable $\tau
_{t}^{m,m+1}=\alpha _{t}^{m+1}\ast (\alpha _{t}^{m})^{-1}$, $1\leq m\leq M-1$%
, and $I(\cdot ;\cdot )$ is the mutual information of its (scalar or
vectorial) arguments. By definition, $C(\mathrm{\alpha }^{1},...,\mathrm{%
\alpha }^{M})\geq 0$. There are other equivalent definitions \cite%
{Monetti2013,Monetti2013B}. For example, in the two-dimensional case,%
\begin{eqnarray*}
C(\mathrm{\alpha }^{1},\mathrm{\alpha }^{2}) &=&\min \{I(\alpha
_{t}^{1};T_{\alpha _{t}^{1},\alpha _{t}^{2}}),I(\alpha _{t}^{2};T_{\alpha
_{t}^{1},\alpha _{t}^{2}})\} \\
&=&\min \{H(\alpha _{t}^{1}),H(\alpha _{t}^{2})\}-H(\alpha _{t}^{1},\alpha
_{t}^{2})+H(T_{\alpha _{t}^{1},\alpha _{t}^{2}}) \\
&=&H(T_{\alpha _{t}^{1},\alpha _{t}^{2}})-\max \{H\left( \alpha
_{t}^{1}\right\vert \alpha _{t}^{2}),H\left( \alpha _{t}^{2}\right\vert
\alpha _{t}^{1})\}.
\end{eqnarray*}

The name CCC is justified by the following two properties \cite{Monetti2013}:

\begin{description}
\item[(i)] If $\mathrm{\alpha }^{1}=\mathrm{\alpha }^{2}=...=\mathrm{\alpha }%
^{M}$, then $C(\mathrm{\alpha }^{1},...,\mathrm{\alpha }^{M})=0$.

\item[(ii)] If $\mathrm{\alpha }^{1},\mathrm{\alpha }^{2},...,\mathrm{\alpha 
}^{M}$ are independent and uniformly distributed, then $C(\mathrm{\alpha }%
^{1},...,\mathrm{\alpha }^{M})=0$.
\end{description}

So, $C(\mathrm{\alpha }^{1},...,\mathrm{\alpha }^{M})>0$ in
\textquotedblleft nontrivial\textquotedblright\ situations. Among other
properties of $C(\mathrm{\alpha }^{1},...,\mathrm{\alpha }%
^{M})$, we mention two more here:

\begin{description}
\item[(iii)] \textit{Invariance} \cite{Monetti2013}. $C(\mathrm{\alpha }^{1},...,\mathrm{\alpha 
}^{M})$ is invariant under permutations of $\mathrm{\alpha }^{1},...,\mathrm{%
\alpha }^{M}$.

\item[(iv)] \textit{Monotonicity} \cite{Monetti2013}. $C(\mathrm{\alpha }^{1},...,\mathrm{%
\alpha }^{k},...,\mathrm{\alpha }^{M})<C(\mathrm{\alpha }^{1},...,\mathrm{%
\alpha }^{k-1},\mathrm{\alpha }^{k+1},...,\mathrm{\alpha }^{M})$ is only
possible when $H(\mathrm{\alpha }^{k})$ is the unique minimum of all the
entropies $H(\mathrm{\alpha }^{m})$, $1\leq m\leq M$; otherwise, $C(\mathrm{%
\alpha }^{1},...,\mathrm{\alpha }^{k},...,\mathrm{\alpha }^{M})\geq C(%
\mathrm{\alpha }^{1},...,\mathrm{\alpha }^{k-1},\mathrm{\alpha }^{k+1},...,%
\mathrm{\alpha }^{M})$.
\end{description}

Furthermore, numerical simulations \cite{Monetti2013B} with ordinal
representations of time series indicate that $C(\mathrm{\alpha }^{1},...,%
\mathrm{\alpha }^{M})$ is monotonically decreasing with respect to the delay
time\ \textrm{T} (see Equation (\ref{ord patt 2})), i.e., $C(\mathrm{\alpha }%
^{1},...,\mathrm{\alpha }^{M})\searrow 0$ when $\mathrm{T}\rightarrow \infty 
$. This property has been applied in the so-called \textit{dimensional
reduction of the conditional mutual information}. In its lowest dimensional
version, this result states that 
\begin{equation}
I(\beta _{t+\Lambda };\alpha _{t}\left\vert \beta _{t}\right) =I(T_{\beta
_{t+\Lambda },\beta _{t}};T_{\alpha _{t},\beta _{t}}),  \label{Dim Red TE}
\end{equation}%
provided that $H(\beta _{t})\leq H(\alpha _{t})$, and $C($\textrm{$\beta $}$%
_{\Lambda },\mathrm{\beta },\mathrm{\alpha })=0$, which, as said above, can
be approximately achieved by fine-tuning the delay time. Here, (i) $\Lambda\geq 1$,  (ii) the conditional mutual
information on the lhs of (\ref{Dim Red TE}) is, by definition, the \textit{%
symbolic transfer entropy} \cite{Staniek2008} from the process $\mathrm{%
\alpha }$ to the process \textrm{$\beta $} with coupling delay $\Lambda $, 
\begin{equation}
I(\beta _{t+\Lambda };\alpha _{t}\left\vert \beta _{t}\right) =:\mathrm{TE}_{%
\mathrm{\alpha }\rightarrow \mathrm{\beta }}(\Lambda ),  \label{TE}
\end{equation}%
and (iii) the mutual information between transcripts on the rhs of (\ref{Dim
Red TE}) is called the \textit{transcript mutual information}
with coupling delay $\Lambda $, 
\begin{equation}
I(T_{\beta _{t+\Lambda },\beta _{t}};T_{\alpha _{t},\beta _{t}})=:\mathrm{TMI%
}_{\mathrm{\alpha }\rightarrow \mathrm{\beta }}(\Lambda ).  \label{TMI}
\end{equation}%
The transfer entropy $\mathrm{TE}_{\mathrm{\alpha }\rightarrow \mathrm{\beta }}(\Lambda )$ is a
measure of the information transfer from $\mathrm{\alpha }$ to \textrm{$%
\beta $} (usually $\Lambda =1$). It is easy to check that $\mathrm{TE}_{%
\mathrm{\alpha }\rightarrow \mathrm{\beta }}(\Lambda )=\mathrm{TMI}_{\mathrm{%
\alpha }\rightarrow \mathrm{\beta }}(\Lambda )=0$ if $\mathrm{\alpha }=%
\mathrm{\beta }$, e.g., when the underlying dynamical systems are fully
synchronized. See Amig\'{o} et al. \cite{Amigo2016} for the higher
dimensional version of equation (\ref{Dim Red TE}).

Among the applications of the CCC and the dimensional reduction (\ref{Dim
Red TE}) of the symbolic transfer entropy via transcripts, let us mention a
few. So, the directionality indicator%
\begin{equation}
\Delta TI_{\mathrm{\alpha }\rightarrow \mathrm{\beta }}=\mathrm{TMI}_{%
\mathrm{\alpha }\rightarrow \mathrm{\beta }}(\Lambda )-\mathrm{TMI}_{\mathrm{%
\beta }\rightarrow \mathrm{\alpha }}(\Lambda )  \label{Delta TI}
\end{equation}%
has been used as an information directionality index between (i) the lateral
geniculate nucleus of the thalamus and the visual cortex infragranular
layers \cite{Amigo2015} and (ii) heart rate and blood pressure with medical
air breathing or oxygen breathing \cite{Amigo2016}, as well as to study the
information directionality between two bidirectionally delay-coupled
logistic maps \cite{Amigo2016}.

Transcript mutual information has been used as a causality index with ice core
atmospheric data, as well as brain data during a visually cued, two-choice
arm reaching task \cite{Hirata2016}.

The coupling complexity coefficient $C(\mathrm{\alpha },\mathrm{\beta })$
and the directionality indicator (\ref{Delta TI}) have been used to study
the dynamics of two, unidirectionally coupled H\'{e}non maps and R\"{o}ssler
oscillators, as well as to characterize brain-wide interactions during
wakefulness and sleep \cite{Adams2025B}.

%%%%%%%%%%%%%%%%%%%%%%%%%%%%%%%%%%%%%%%%%%%%%

\subsection{Order classes}

\label{sec53}

The order (or period) of a symbol $a\in \mathcal{G}$, \textrm{ord}$(a)$, is
the minimum positive integer $m$ such that $a^{m}=e$, the identity of $%
\mathcal{G}$. Three basic properties of \textrm{ord}$(a)$ are the following 
\cite{Herstein1996}:

\begin{description}
\item[(O1)] The only element of $\mathcal{G}$ with order $1$ is the identity 
$e$.

\item[(O2)] \textrm{ord}$(a)$ is a divisor of $\left\vert \mathcal{G}%
\right\vert $ for all $a\in \mathcal{G}$. In symbols: \textrm{ord}$%
(a)|\left\vert \mathcal{G}\right\vert $.

\item[(O3)] \textrm{ord}$(a)=\mathrm{ord}(a^{-1})$ for all $a\in \mathcal{G}$%
.
\end{description}

We define the \textit{order class} $\mathcal{C}_{m}$ as the subset of $%
\mathcal{G}$ that comprises all elements of order $m$,%
\begin{equation}
\mathcal{C}_{m}=\{a\in \mathcal{G}:\mathrm{ord}(a)=m\,\},  \label{C_m}
\end{equation}%
where $m\geq 1$ divides $\left\vert \mathcal{G}\right\vert $. In particular, 
$\mathcal{C}_{1}=\{e\}$ and $\mathcal{C}_{2}=\{a\in \mathcal{G}:a=a^{-1}\}$.

The sets $\mathcal{C}_{m}$ build a partition of $\mathcal{G}$. If $\left\vert \mathcal{G}\right\vert $ is a prime number $p$,
then $\mathcal{G}=\mathcal{C}_{1}\cup \mathcal{C}_{p}$, i.e., $\mathcal{C}%
_{p}=\mathcal{G}\backslash \{e\}$, which means that the group is cyclic
(generated by the powers of any element $a\neq e$, with $a^{0}:=e$).

By properties (T1) in Section \ref{sec4} and (O3) above, the order of the
transcripts $T_{a,b}$ are invariant under the exchange of the source $a$ and
the target $b$, i.e., \textrm{ord}$(T_{a,b})=\,$\textrm{ord}$(T_{b,a})$. The
order of the transcript $T_{a,b}$ was proposed as a measure of dissimilarity
between $a$ and $b$ by Monetti et al. \cite{Monetti2009}. In Remark \ref{Remark order=distance}, equation (\ref{order = distance2}), we will see that \textrm{ord}$(T_{a,b})-1$ is actually a distance when $\mathcal{G}=\textrm{Sym}(3)$. 

In time series analysis of coupled $\mathcal{G}$-valued time series $\mathrm{%
\alpha }=(\alpha _{t})_{t\geq 0}$ and $\mathrm{\beta }=(\beta _{t})_{t\geq
0} $, we are interested in the order classes of their transcription $\mathrm{%
T}_{\mathrm{\alpha },\mathrm{\beta }_{\Lambda }}=(\tau _{t}(\Lambda
))_{t\geq 0}=(\beta _{t+\Lambda }\ast \alpha _{t}^{-1})_{t\geq 0}$. Again,
the order classes of the transcripts are invariant under the exchange of the
source $\mathrm{\alpha }$ and the target $\mathrm{\beta }$.

More interesting from a physical point of view, it was observed by Monetti
et al. \cite{Monetti2009} that the dynamics of coupled nonlinear systems can
lead to the extinction of transcript order classes, a phenomenon called
saturation. As a trivial example, if $\alpha $ and $\beta $ are algebraic
representations of a fully synchronized drive-response system, then $%
\mathcal{C}_{m}=\emptyset $ for all $m\neq 1$ because then $\mathrm{\beta }=%
\mathrm{\alpha }$ and, hence, $\tau _{t}=\alpha _{t}\ast \alpha _{t}^{-1}=e$
for all $t\geq 0$. Non-trivial examples can arise in case of generalized
synchronzation, as we will see in Section \ref{sec72}.

Although any transcript-based entropic measure is in principle sensitive to
saturation in coupled dynamics, there are tools tailored to this end. For
example, the transcript probabilities $p(\tau )$ and $p^{\mathrm{ind}}(\tau
) $ can be lumped by order class. In such a case:

(a) $P_{\mathbf{\tau }}=\{p(\tau _{t}):\tau _{t}\in \mathcal{G}\}$ changes
to $P_{\tau }^{\mathcal{C}}=\{p_{\mathcal{C}_{m}}:m|\left\vert \mathcal{G}%
\right\vert \}$, where 
\begin{equation}
p_{\mathcal{C}_{m}}=\sum\nolimits_{\tau \in \mathcal{C}_{m}}p(\tau ).
\label{p_C_m}
\end{equation}

(b) $P_{\tau }^{\mathrm{ind}}$ changes to $P_{\tau }^{\mathcal{C},\mathrm{ind%
}}=\{p_{\mathcal{C}_{m},\mathrm{ind}}:$ $m|\left\vert \mathcal{G}\right\vert
\}$, where 
\begin{equation}
p_{\mathcal{C}_{m},\mathrm{ind}}=\sum\nolimits_{\tau \in \mathcal{C}_{m}}p_{%
\mathrm{ind}}(\tau )  \label{p_C_m_ind}
\end{equation}%
and $p^{\mathrm{ind}}(\tau )$ is given in equation (\ref{Pind}).

The divergences $D_{KL}^{\mathrm{sym}}(P_{\tau }^{\mathcal{C}}\parallel
P_{\tau }^{\mathcal{C}\text{,}\mathrm{ind}})$ and the probabilities $p_{%
\mathcal{C}_{m}}$ (with ordinal patterns of lengths 6 and 7) have been used 
\cite{Monetti2009} to characterize the synchronization regimes of a
bidirectionally coupled R\"{o}ssler-R\"{o}ssler system. A similar study has
also been performed \cite{Bunk2013}\ with the frequencies of order classes and
patterns of length 6.

Likewise, the divergence $D_{JS}\left( P_{\tau }^{\mathcal{C}}\right\Vert
P_{\tau }^{\mathcal{C},\mathrm{ind}})$ and the probabilities $p_{\mathcal{C}%
_{m}}$ (with ordinal patterns of length 4) have been used \cite{Adams2025B}
to study the dynamics of two, unidirectionally coupled H\'{e}non maps and R%
\"{o}ssler oscillators, as well as to characterize brain-wide interactions
during wakefulness and sleep.

%%%%%%%%%%%%%%%%%%%%%%%%%%%%%%%%%%%%%%%%%%%%%%

\subsection{Distances in the symmetric groups}

\label{sec61}

An edit distance between two strings of symbols is defined as the minimum
number of allowed edit operations (insertions, deletions, substitutions,
transpositions) to transform one string into the other. Since permutations,
written in one-line form, are symbolic strings without repeated symbols, any
suitable edit distance can be used to measure the distance between them, as
proposed by S\"{o}rensen \cite{Sorensen2007}.

The perhaps most important edit distances for the symmetric groups $\mathrm{%
Sym}(L)$ are the Cayley distance \cite{Nguyen2024} and the Kendall (tau-)
distance \cite{Kendall1938}, where the allowed operations are transpositions
and adjacent transpositions of symbols, respectively. Specifically,
transpositions are cycles of length 2, i.e., a transposition $(ij)$ is a
permutation $\mathbf{t}_{ij}\in \mathrm{Sym}(L)$ such that $\mathbf{t}%
_{ij}(i)=j$, $\mathbf{t}_{ij}(j)=i$, and $\mathbf{t}_{ij}(k)=k$ for all $%
k\neq i,j$. If $\mathbf{r}=(r_{1},...,r_{L})$, then%
\begin{equation}
\mathbf{t}_{ij}\ast \mathbf{r}%
=(r_{1},...,r_{i-1},r_{j},r_{i+1},...,r_{j-1},r_{i},r_{j+1},...r_{L}).
\label{transposition}
\end{equation}%
If $\left\vert i-j\right\vert =1$, then $\mathbf{t}_{ij}$ is called an 
\textit{adjacent transposition}. Unlike factorization of permutations
into disjoint cycles, the factorization of permutations into adjacent
transpositions (and, hence, into transpositions) is not unique, although the
minimal number of factors is. For example, $321=(12)(23)(12)=(23)(12)(23)$.

\begin{definition}
\label{DefDistances}Let $\mathbf{r},\mathbf{s}\in \mathrm{Sym}(L)$. The 
\emph{Cayley} (resp. Kendall) distance between $\mathbf{r}$ and $%
\mathbf{s}$, denoted by $d_{C}(\mathbf{r},\mathbf{s})$ 
(resp. $d_{K}(\mathbf{r},\mathbf{s})$), is defined as the
minimum number of transpositions (resp. adjacent transpositions) 
needed to transform $\mathbf{r}$ into $\mathbf{s}$.
\end{definition}

The proof that $d_{C}$ and $d_{K}$ are distances in the axiomatic sense is
straightforward and can be found (using graph-theoretical arguments for the
triangular inequality) in Amig\'{o} and Dale \cite{Amigo2025}.

Figure \ref{figure1} shows the adjacency graph of $\mathrm{Sym}(4)$, i.e.,
each node corresponds to a permutation $\mathbf{r}\in \mathrm{Sym}(4)$ and
two nodes are connected by a link if they differ by an adjacent
transposition. So the distance $d_{K}(\mathbf{r},\mathbf{s})$ is the length
(number of links) of the shortest path between $\mathbf{r}$ and $\mathbf{s}$%
. For a 3D rendering of the adjacency group of $\mathrm{Sym}(4)$, see the
permutohedron \cite{Permutohedron}. The adjacency graph of $\mathrm{Sym}(3)$
is just a cycle \cite{Amigo2025}.

\begin{figure*}[tbh]
\begin{center}
\includegraphics[width=120mm]{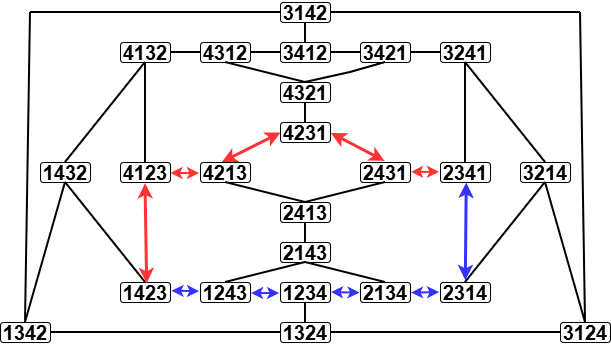}
\end{center}
\caption{Adjacency graph of $\mathrm{Sym}(4)$. A link between two
permutations means that the Kendall distance (minimum number of adjacent
transpositions that transforms one of then into the other) is 1. The figure
highlights two (out of five) shortest paths between the permutations 1423
and 2341, marked with bidirectional arrows (in blue and red online), showing
that $d_{K}(1423,2341)=5$.}
\label{figure1}
\end{figure*}

The distances $d_{C}(\mathbf{r},\mathbf{s})$ and $d_{K}(\mathbf{r},\mathbf{s}%
)$ can be calculated using the following formulas, both involving the
transcript $T_{\mathbf{r},\mathbf{s}}=\mathbf{s}\ast \mathbf{r}^{-1}$. In
view of Definition \ref{DefDistances}, it is not surprising that the
transcript $T_{\mathbf{r},\mathbf{s}}$ is related to $d_{C}(\mathbf{r},%
\mathbf{s})$ and $d_{K}(\mathbf{r},\mathbf{s})$: By equation (\ref{T(a,b)}), 
$T_{\mathbf{r},\mathbf{s}}\ast \mathbf{r}=\mathbf{s}$, so $T_{\mathbf{r},%
\mathbf{s}}$ is the element of $\mathcal{G}$ that transforms $\mathbf{r}$
into $\mathbf{s}$.

\begin{proposition}
\label{PropComput dC dK}(a) Let $\mathbf{u}=(u_{1},...,u_{L})\in \mathrm{Sym}%
(L)$ and $C(\mathbf{u})$ be the number of cycles (including $1$-cycles) in
the cycle factorization of the permutation $\mathbf{u}$. Then,%
\begin{equation}
d_{C}(\mathbf{r},\mathbf{s})=L-C(\mathbf{s}\ast \mathbf{r}^{-1})=L-C(T_{%
\mathbf{r},\mathbf{s}})  \label{Comput d_C}
\end{equation}%
for all $\mathbf{r},\mathbf{s}\in \mathrm{Sym}(L).$

(b) Let $I(\mathbf{u})$ be the number of inversions in the permutation $%
\mathbf{u}$, i.e., the number of ordered pairs $(u_{i},u_{j})$, $1\leq
i<j\leq L$, such that $u_{i}>u_{j}$. Then%
\begin{equation}
d_{K}(\mathbf{r},\mathbf{s})=I(\mathbf{s}\ast \mathbf{r}^{-1})=I(T_{\mathbf{r%
},\mathbf{s}})  \label{Comput d_K}
\end{equation}%
for all $\mathbf{r},\mathbf{s}\in \mathrm{Sym}(L)$.
\end{proposition}

Of course, $d_{C}(\mathbf{r},\mathbf{s})=d_{C}(\mathbf{s},\mathbf{r})$
amounts to $C(T_{\mathbf{r},\mathbf{s}})=C(T_{\mathbf{s},\mathbf{r}})$, and $%
d_{K}(\mathbf{r},\mathbf{s})=d_{K}(\mathbf{s},\mathbf{r})$ to $I(T_{\mathbf{r%
},\mathbf{s}})=I(T_{\mathbf{s},\mathbf{r}})$. By definition,%
\begin{equation}
d_{C}(\mathbf{r},\mathbf{s})\leq d_{K}(\mathbf{r},\mathbf{s}).
\label{d_C <= d_K}
\end{equation}
Furthermore, from equations (\ref{Comput d_C}) and (\ref{Comput d_K}) it
follows 
\begin{equation}
d_{C}(\mathbf{r},\mathbf{s})\in \{0,1,...,d_{C,\max }:=L-1\},
\label{Range d_C}
\end{equation}%
and 
\begin{equation}
d_{K}(\mathbf{r},\mathbf{s})\in \left\{ 0,1,...,d_{K,\max }:=\frac{L(L-1)}{2}%
\right\} ,  \label{Range d_K}
\end{equation}%
Therefore, $d_{K,\max
}>d_{C,\max }$, except for $L=2$, which generally makes $d_{K}$ a better
choice\ in applications.

An important property of $d_{C}$ and $d_{K}$ is \textit{right-invariance%
}: Since 
\begin{equation}
T_{\mathbf{e},\mathbf{s}\ast \mathbf{r}^{-1}}=\mathbf{s}\ast \mathbf{r}%
^{-1}=T_{\mathbf{r},\mathbf{s}},  \label{T=T}
\end{equation}%
we derive from equation (\ref{Comput d_C})-(\ref{Comput d_K}) that%
\begin{equation}
d_{C,K}(\mathbf{r},\mathbf{s})=d_{C,K}(\mathbf{e},\mathbf{s}\ast \mathbf{r}%
^{-1}),  \label{right invariance}
\end{equation}%
where $\mathbf{e}=12\ldots L$ is the identity permutation and, here and
below, we use the shorthand $d_{C,K}$ to refer to both the Cayley and
Kendall distances. Define the usual norm%
\begin{equation}
\left\Vert \mathbf{r}\right\Vert _{C,K}:=d_{C,K}(\mathbf{e},\mathbf{r})
\label{norm}
\end{equation}%
in the metric spaces $(\mathrm{Sym}(L),d_{C,K})$ to write equation (\ref%
{right invariance}) in the appealing form%
\begin{equation}
d_{C,K}(\mathbf{r},\mathbf{s})=\left\Vert T_{\mathbf{r},\mathbf{s}%
}\right\Vert_{C,K} .  \label{norm2}
\end{equation}

According to equation (\ref{right invariance}), all possible distances $%
d_{C,K}(\mathbf{r},\mathbf{s})$ appear in the $\mathbf{e}$-row (i.e., $%
(d_{C,K}(\mathbf{e},\,\mathbf{u}))_{\mathbf{u}\in \mathrm{Sym}(L)}$) of the
distance matrix $(d_{C,K}(\mathbf{r},\mathbf{s}))_{\mathbf{r,s}\in \mathrm{%
Sym}(L)}$.

\begin{example}
\label{ExampleSym(3)}The Cayley and Kendall distance matrices for the group $%
\mathrm{Sym}(3)$, equation (\ref{mult table Sym(3)B}), are the following: 
\begin{equation}
\begin{tabular}{|c||c|c|c|c|c|c|}
\hline
$d_{C}(\mathbf{r},\mathbf{s})$ & $123$ & $132$ & $213$ & $231$ & $312$ & $%
321 $ \\ \hline\hline
$123$ & $0$ & $1$ & $1$ & $2$ & $2$ & $1$ \\ \hline
$132$ & $1$ & $0$ & $2$ & $1$ & $1$ & $2$ \\ \hline
$213$ & $1$ & $2$ & $0$ & $1$ & $1$ & $2$ \\ \hline
$231$ & $2$ & $1$ & $1$ & $0$ & $2$ & $1$ \\ \hline
$312$ & $2$ & $1$ & $1$ & $2$ & $0$ & $1$ \\ \hline
$321$ & $1$ & $2$ & $2$ & $1$ & $1$ & $0$ \\ \hline
\end{tabular}
\label{d_C(Sym(3))}
\end{equation}%
and 
\begin{equation}
\begin{tabular}{|c||c|c|c|c|c|c|}
\hline
$d_{K}(\mathbf{r},\mathbf{s})$ & $123$ & $132$ & $213$ & $231$ & $312$ & $%
321 $ \\ \hline\hline
$123$ & $0$ & $1$ & $1$ & $2$ & $2$ & $3$ \\ \hline
$132$ & $1$ & $0$ & $2$ & $3$ & $1$ & $2$ \\ \hline
$213$ & $1$ & $2$ & $0$ & $1$ & $3$ & $2$ \\ \hline
$231$ & $2$ & $3$ & $1$ & $0$ & $2$ & $1$ \\ \hline
$312$ & $2$ & $1$ & $3$ & $2$ & $0$ & $1$ \\ \hline
$321$ & $3$ & $2$ & $2$ & $1$ & $1$ & $0$ \\ \hline
\end{tabular}
\label{d_K(Sym(3))}
\end{equation}%
As shown in equations (\ref{d_C <= d_K})-(\ref{Range d_K}), $d_{C}(\mathbf{r}%
,\mathbf{s})\leq d_{K}(\mathbf{r},\mathbf{s})$ for all $\mathbf{r},\mathbf{%
s\in }\mathrm{Sym}(3)$, $d_{C}(\mathbf{r},\mathbf{s})\in \{0,1,2\}$, and $%
d_{K}(\mathbf{r},\mathbf{s})\in \{0,1,2,3\}$. All these possible distances
appear in the first row (row $123$) of the corresponding distance matrix.
The distance matrices for $\mathrm{Sym}(4)$ can be found in the Appendix of
Amig\'{o} and Dale \cite{Amigo2025}.
\end{example}

%\begin{remark}
%In Section \ref{sec72} we will use the fact that $d_{C}%(\mathbf{r},\mathbf{s})=\mathrm{ord}(T_{\mathbf{r},\mathbf{s}})-1$ for $%
%\mathrm{Sym}(3)$, i.e., the order of the transcript (Section \ref{sec53}) between two ordinal 3-patterns minus one is their Cayley distance (equation (\ref{d_C(Sym(3))}))
%\end{remark}

Empirical probability distributions of the distances
\begin{equation}
\mathrm{dist}(\mathbf{r%
}_{t},\mathbf{s}_{t}):=d_{C,K}(\mathbf{r}_{t},\mathbf{s}_{t})_{1\leq t\leq
N-L+1}
\label{W=1}
\end{equation} 
have been used \cite{Amigo2025} to characterize generalized
synchronization in a drive-response system composed of two H\'{e}%
non maps, which has two different synchronization regimes: weak and strong.
To this end, the first components of the driver and responder were
discretized with ordinal patterns of lengths $L=4,5$, providing the time
series $\alpha$=$(\mathbf{r}_{t})_{1\leq t\leq N}$ and $\beta$=$(\mathbf{s}_{t})_{1\leq t\leq
N}$. The results showed that weak synchronization characterizes by
forbidden short distances, while strong synchronization does by forbidden long distances. Since $d_{K}$ has a greater range than $d_{C}$ for the same $L\ge 3$, the former has greater differentiating power
than the latter.

More generally, sliding windows $\mathbf{r}_{t}^{W}=(\mathbf{r}_{t},\mathbf{r%
}_{t+1},...,\mathbf{r}_{t+W-1})$ and $\mathbf{s}_{t}^{W}=(\mathbf{s}_{t},%
\mathbf{s}_{t+1},...,\mathbf{s}_{t+W-1})$ of size $W\geq 2$ have also been
used \cite{Amigo2025} for characterizing generalized synchronization in the above 
drive-response system, their similarity being measured by the $l_{p}$-distances

\begin{equation}
\mathrm{dist}_{p}(\mathbf{r}_{t}^{W},\mathbf{s}_{t}^{W})=\left(
\sum_{k=0}^{W-1}\mathrm{dist}(\mathbf{r}_{t+k},\mathbf{s}_{t+k})^{p}\right)
^{1/p}  \label{l_p dist}
\end{equation}%
for $1\leq p<\infty $, and%
\begin{equation}
\mathrm{dist}_{\infty }(\mathbf{r}_{t}^{W},\mathbf{s}_{t}^{W})=\max \left\{ 
\mathrm{dist}(\mathbf{r}_{t+k},\mathbf{s}_{t+k}):0\leq k\leq W-1\right\} .
\label{Chebyshev dist}
\end{equation}%
In this case, one ends up with an integer (for $p=1,\infty $) or real (for $%
1<p<\infty $) sequence%
\begin{equation}
(\mathrm{dist}_{p}(\mathbf{r}_{t}^{W},\mathbf{s}_{t}^{W}))_{1\leq t\leq
N-W+1}  \label{dist_p TS}
\end{equation}%
that contains metric information about the time series $(\mathbf{r}_{t})$
and $(\mathbf{s}_{t})$. Distances $\mathrm{dist}_{p}$ with smaller
parameters $p$ (say, the Manhattan distance $\mathrm{dist}_{1}$ and the
Euclidean distance $\mathrm{dist}_{2}$) have greater differentiating
power due to the monotony property of the $p$-norms ($\left\Vert \cdot
\right\Vert _{p}\geq \left\Vert \cdot \right\Vert _{p^{\prime }}$ for $1\leq
p\leq $ $p^{\prime }\leq \infty $). The results supported the results
obtained with simultaneous ordinal patterns ($W=1$), equation (\ref{W=1}).

%%%%%%%%%%%%%%%%%%%%%%%%%%%%%%%%%%%%%%%%%

\subsection{Distances in general groups via permutations}

\label{sec62}

Given a group $\mathcal{G}$, a finite set $S=\{s_{1},...,s_{n}\}\subset 
\mathcal{G}$ is a \textit{generating set} (or generator) of $\mathcal{G}$ if
every $a\in \mathcal{G}$ can be written as a finite product of elements of $%
S $ and their inverses. The distance (or \textit{word metric}) $d_{S}(a,b)$
between two elements $a,b\in \mathcal{G}$ is defined \cite{Herstein1996} as
the minimum number of elements from the generating set $S$ needed to
transform $a$ into $b$. In particular, if $\mathcal{G}=$ $\mathrm{Sym}(L)$,
then the Cayley distance $d_{C}(\mathbf{r},\mathbf{s})$ is the distance $%
d_{S}$ with respect to the generating set of all transpositions, while the
Kendall distance $d_{K}(\mathbf{r},\mathbf{s})$ is the distance $d_{S}$ with
respect to the generating set of all adjacent transpositions.

For groups $\mathcal{G}$ other than $\mathrm{Sym}(L)$, a different approach
was proposed in Amig\'{o} and Dale \cite{Amigo2025} which dispenses with
generating sets and, hence, with the search for minimal factorizations of
transcripts using generating elements. This method is based on Cayley's
Theorem (Theorem \ref{ThmCayley}), which allows embedding any group $(%
\mathcal{G},\cdot )$ into the symmetric group $\mathrm{Sym}(\mathcal{G})$.

\begin{definition}
\label{DefDist(a,b)}Let $\Phi :\mathcal{G}\rightarrow \mathrm{Sym}(\mathcal{G%
})$ be the Cayley embedding for a finite group $(\mathcal{G},\cdot )$. Then, 
$D_{C,K}^{(\Phi )}$ is the distance in $\mathcal{G}$ defined as 
\begin{equation}
D_{C,K}^{(\Phi )}(a,b)=d_{C,K}(\Phi (a),\Phi (b))  \label{dist(a,b)}
\end{equation}%
for all $a,b\in \mathcal{G}$.
\end{definition}

Hence, $D_{C,K}^{(\Phi )}$ has the same properties as $d_{C,K}$. In particular:

\begin{itemize}
\item For $a,b\in \mathcal{G}$,%
\begin{equation}
D_{C}^{(\Phi )}(a,b)\leq D_{K}^{(\Phi )}(a,b),  \label{D_C<=D_K}
\end{equation}%
where 
\begin{equation}
D_{C}^{(\Phi )}(a,b)\in \left\{ 0,1,...,D_{C,\max }^{(\Phi )}:=\left\vert 
\mathcal{G}\right\vert -1\right\} ,\;  \label{DC max}
\end{equation}%
and 
\begin{equation}
D_{K}^{(\Phi )}(a,b)\in \left\{ 0,1,...,D_{K,\max }^{(\Phi )}:=\frac{%
\left\vert \mathcal{G}\right\vert (\left\vert \mathcal{G}\right\vert -1)}{2}%
\right\} .  \label{DK max}
\end{equation}

\item If $e$ is the identity of $\mathcal{G}$, then 
\begin{equation}
D_{C,K}^{(\Phi )}(a,b)=D_{C,K}^{(\Phi )}(e,T_{a,b}).  \label{Left inv D 2}
\end{equation}
\end{itemize}

%\begin{remark}
In the case $\mathcal{G}=\mathrm{Sym}(L)$, Section \ref{sec61}, the
distances $d_{C,K}(\mathbf{r},\mathbf{s})$ take on all integer values
ranging from $0$ to their respective maxima: $d_{C,\max }=L-1$ (equation (%
\ref{Range d_C})), and $d_{K,\max }=L(L-1)/2$ (equation (\ref{Range d_K})).
However, this does not happen with $D_{C,K}^{(\Phi )}(a,b)$ because $\Phi (%
\mathcal{G})$ is a subgroup of cardinality $\left\vert \mathcal{G}%
\right\vert $ of the group $\mathrm{Sym}(\mathcal{G})$, whose cardinality is 
$\left\vert \mathcal{G}\right\vert !$, so not all possible distances can be
realized (unless $\left\vert \mathcal{G}\right\vert =2$). We call
\textquotedblleft the gaps in $D_{C,K}^{(\Phi )}$%
\textquotedblright\ the values in $\{0,1,...,D_{C,K,\max }^{(\Phi )}\}$ missing
from the distance matrix $(D_{C,K}^{(\Phi )}(a,b))_{a,b\in 
\mathcal{G}}$; otherwise, they are called allowed or admissible distances.
By equation (\ref{Left inv D 2}), the admissible distances for $%
D_{C,K}^{(\Phi )}$ can be read in the row $(D_{C,K}^{(\Phi )}(e,c))_{c\in 
\mathcal{G}}$ of the distance matrix. Hence, the maximum number of
admissible distances for $D_{C,K}^{(\Phi )}$ is $\left\vert \mathcal{G}%
\right\vert $.
%\end{remark}

According to Theorem \ref{ThmCayleyViaTranscripts}, the Cayley embedding $%
\Phi :\mathcal{G}\rightarrow \mathrm{Sym}(\mathcal{G})$ can be implemented
via the left translations $a\mapsto \tilde{T}_{a}:=\tilde{T}(a,\cdot )$,
where $\tilde{T}(a,b)=a^{-1}\cdot b$ is the conjugate transcription mapping, equation (\ref{T(a,b) conjugate}). In this case, we write $%
D_{C,K}^{(\tilde{T})}$ instead of $D_{C,K}^{(\Phi )}$ to specify that the
embedding $\Phi $ is implemented by conjugate transcripts.

\begin{example}
\label{Example Dist for Klein group} Consider again the Klein four-group $\mathcal{K}$, Example \ref{ExampleKlein}, and the Cayley isomorphism  $\Phi(r)=\tilde{T}_{r}:\mathcal{K} \to \textrm{Sym}(4)$. To calculate the Kendall distances $D_{K}^{(\tilde{T})}(r,s)$ $=d_{K}(\tilde{T}_{r},\tilde{T}_{s})$, we are going to take advantage of the adjacency graph of the embedding symmetric group $\mathrm{Sym}(4)$, Figure \ref{figure1}. To this end, we use the encoding (\ref{encoding}), locate the isomorphic copies 
\begin{equation*}
  \tilde{T}_{e}=1234, \; \tilde{T}_{a}=2143, \;
  \tilde{T}_{b}=3412, \; \tilde{T}_{c}=4321,
  \label{encoding Klein}
\end{equation*}
(see equation (\ref{encoding Klein2})) in Figure \ref{figure1},
%\begin{equation}
% \tilde{T}_{e}=1234, \; \tilde{T}_{a}=2143, \; \tilde{T}_{b}=3412, \; %\tilde{T}_{c}=4321 
%\end{equation}
%Figure \ref{figure1} to calculate  To this end, encode the
%elements $e,a,b,c$ as $1,2,3,4$, respectively, and locate their isomorphic
%copies $\tilde{T}_{e}=1234$, $\tilde{T}_{a}=2143$, $\tilde{T}_{b}=3412$, and 
%$\tilde{T}_{c}=4321$ in the adjacency graph of the embedding symmetric group 
%$\mathrm{Sym}(4)$, Figure \ref{figure1}, and read there the distance $d_{K}(%
%\tilde{T}_{r},\tilde{T}_{s})$ between any two copies $\tilde{T}_{r}$, $%
%\tilde{T}_{s}$ of $r,s\in \mathcal{K}$. 
and read there the distances between them. The result is%
\begin{equation}
\begin{tabular}{|c||c|c|c|c|}
\hline
$D_{K}^{(\tilde{T})}$ & $e$ & $a$ & $b$ & $c$ \\ \hline\hline
$e$ & $0$ & $2$ & $4$ & $6$ \\ \hline
$a$ & $2$ & $0$ & $6$ & $4$ \\ \hline
$b$ & $4$ & $6$ & $0$ & $2$ \\ \hline
$c$ & $6$ & $4$ & $2$ & $0$ \\ \hline
\end{tabular}
\label{D_K for K}
\end{equation}%
As noted before, all allowed distances can be read in
the row corresponding to $\Phi (e)=\tilde{T}_{e}$ of the distance matrix.
Since $D_{K}^{(\tilde{T})}(r,s)\in \{0,1,...,6\}$ according to equation (\ref%
{DK max}), only the even values (including $0$) are allowed for $D_{K}^{(%
\tilde{T})}$. It is interesting that the nodes $1234$, $2143$, $3412$, and $%
4321$ are located on the central axis of the $\mathrm{Sym}(4)$ adjacency
graph, symmetrically positioned with respect to the center of the graph.
\end{example}

%%%%%%%%%%%%%%%%%%%%%%%%%%%%%%%%%%%%%%%%%%%%%%%%%%

\section{Numerical simulations}

\label{sec7}

In this section we are going to test numerically the performance of some representative tools from Section \ref{sec5} in generalized synchronization detection. To
this end we resort to the same model used in Amig\'{o} and Dale \cite%
{Amigo2025} where, instead, the probability distributions of some Kendall and $l_{p}$-distances were analysed. The model is composed of two unidirectionally
coupled, non-identical H\'{e}non systems \cite{Schiff1996}. The equations of
the driver $X$ are%
\begin{equation}
\left\{ 
\begin{array}{l}
x_{t+1}^{(1)}=1.4-(x_{t}^{(1)}{})^{2}+0.1x_{t}^{(2)} \\ 
x_{t+1}^{(2)}=x_{t}^{(1)}%
\end{array}%
\right.   \label{HenonMapX}
\end{equation}%
and the equations of the responder $Y$ are 
\begin{equation}
\left\{ 
\begin{array}{l}
y_{t+1}^{(1)}=1.4-[Cx_{t}^{(1)}+(1-C)y_{t}^{(1)}{}]y_{t}^{(1)}+0.3y_{t}^{(2)}
\\ 
y_{t+1}^{(2)}=y_{t}^{(1)}%
\end{array}%
\right.   \label{HenonMap2}
\end{equation}%
where $C>0$ is the \textit{coupling constant} or \textit{strength}. The
reason for selecting this model here again is two-fold: (i) it has \textit{%
generalized synchronization} \cite{Rulkov1995,Rosenblum1997,Pikovsky2001} in
the parametric interval $0.50\lesssim C\lesssim 0.60$ and for $C\gtrsim 0.90$
(Figure 3 of Amig\'{o} et al. \cite{Amigo2024}), and (ii) we can compare
some of our results here with the results obtained there by other means and cross-interpret each result with the help of the others. We will refer to the generalized synchronization of the system (\ref{HenonMapX})-(\ref{HenonMap2}) for $0.50\lesssim C\lesssim 0.60$ and $C\gtrsim 0.90$
as \textquotedblleft weak\textquotedblright\ and \textquotedblleft
strong\textquotedblright\ synchronization, respectively.

For a given coupling constant $C$, let $\mathrm{x}=(x_{t}^{(1)})_{1\leq
t\leq N}$ and $\mathrm{y}=(y_{t}^{(1)})_{1\leq t\leq N}$ be two stationary
time series of length $N=10,000$ composed of the first components of the
states $x_{t}=(x_{t}^{(1)},x_{t}^{(2)})$ of the driver and $%
y_{t}=(y_{t}^{(1)},y_{t}^{(2)})$ of the responder, respectively, and
obtained with seeds $(0,0.9)$ and $(0.75,0)$ after discarding an initial
transient. Let $\mathrm{\alpha }=(\mathbf{r}_{t})_{1\leq t\leq N-L+1}$ and $%
\mathrm{\beta }=(\mathbf{s}_{t})_{1\leq t\leq N-L+1}$ be the algebraic
representations of $\mathrm{x}$ and $\mathrm{y}$ using ordinal patterns of
length $L=3$, and $\mathrm{T}_{\mathrm{\alpha },\mathrm{\beta }}=(\mathbf{t}%
_{t})_{1\leq t\leq N-L+1}=(\mathbf{s}_{t}\mathbf{\ast r}_{t}^{-1})_{1\leq
t\leq N-L+1}$ the corresponding transcript time series. The values chosen
for the coupling strength are $0\leq C\leq 1.2$ with $\Delta C=0.05$. For
computer codes to calculate ordinal patterns and tools, see Unakafova and
Keller \cite{Unakafova2013}, Berger et al. \cite{Berger2019}, and Pessa and
Ribeiro \cite{Pessa2021}.

%%%%%%%%%%%%%%%%%%%%%%%%%%%%%%%%%%%%%%

\subsection{Results with the entropy-complexity plane}

\label{sec71}

\begin{figure*}[tbh]
\begin{center}
\includegraphics[width=140mm]{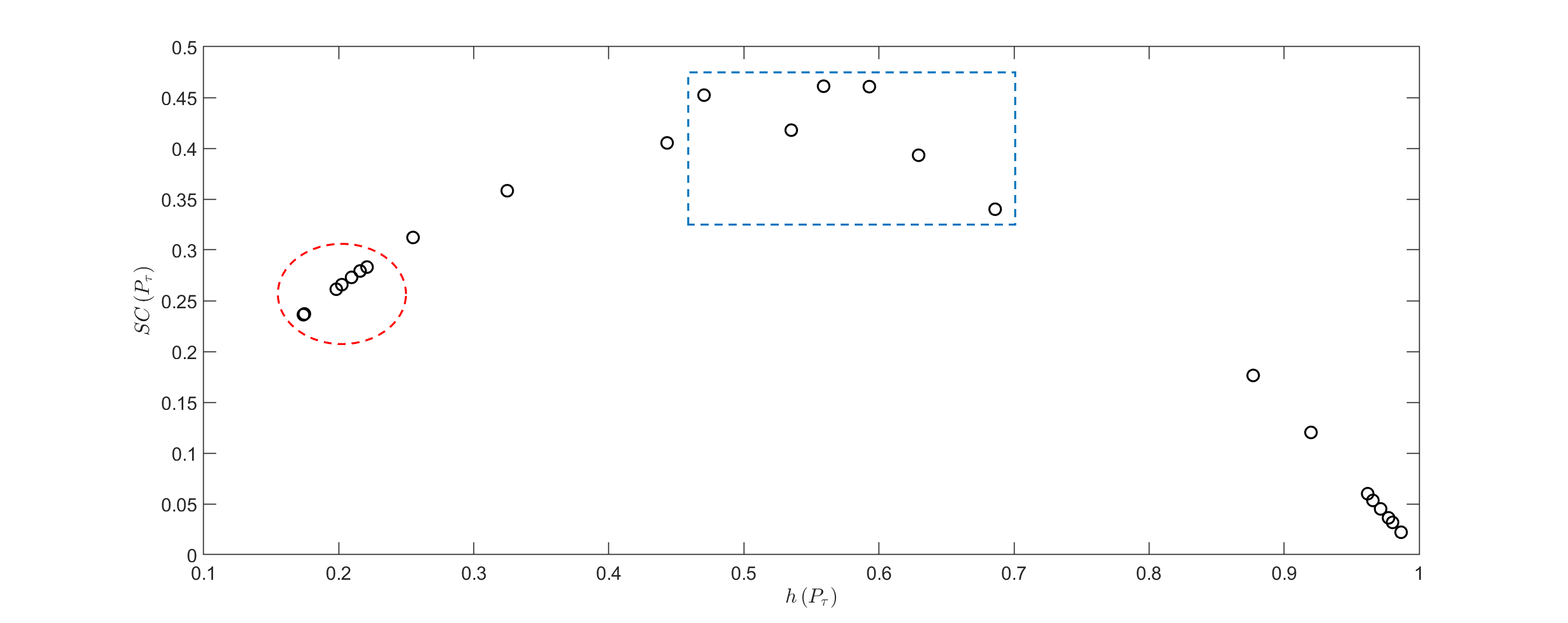}
\end{center}
\caption{Entropy-complexity plane $h(P_{\protect\tau })\times \mathrm{SC}(P_{%
\protect\tau })$ for the H\'{e}non-H\'{e}non system (\protect\ref{HenonMapX}%
)-(\protect\ref{HenonMap2}). Each point on the plane corresponds to a time
series of length-3 transcripts, obtained with a given coupling strength. All
points corresponding to strong synchronization ($C\gtrsim 0.9$) cluster
inside the dashed circle in the lower left of the plane. Points
corresponding to weak synchronization ($0.50\lesssim C\lesssim 0.60$) \ and
nearly weak synchronization ($C=0.40,0.45,0.65$) are distributed inside the
dashed square in the upper center of the plane.}
\label{figure2}
\end{figure*}

Let $P_{\tau }$ be the transcript probability distribution obtained from the
time series $\mathrm{T}_{\mathrm{\alpha },\mathrm{\beta }}$ with coupling
constant $C$. Figure \ref{figure2} shows the values of the normalized
Shannon entropy $h(P_{\tau })$ (equation (\ref{normalized H}) with $\Lambda
=0$ and $\left\vert \mathcal{G}\right\vert =\left\vert \mathrm{Sym}%
(3)\right\vert =6$) and the statistical complexity 
\begin{equation}
\mathrm{SC}(P_{\tau })=D_{JS}(P_{\tau }\parallel U)h(P_{\tau })
\label{SC Results}
\end{equation}%
(equation (\ref{SC_JensenShannon})) on the entropy-complexity plane \cite%
{Rosso2007} for $C=0.05k$ and $0\leq k\leq 24$. Here $U$ is the uniform
distribution over $6$ events. We use base 2 logarithms so that $0\leq 
\mathrm{SC}(P_{\tau })\leq 1$.

The dashed circle in the lower left of Figure \ref{figure2} encloses the
coupled systems with (left to right) $C=0.95$, $1.00$, $0.85$, $0.90$, $1.05$%
, $1.10$, $1.15$, and $1.20$ (where $C=0.85,0.90$ overlap, as well as $%
C=1.15,1.20$), i.e., all strongly synchronized systems ($C\geq 0.90$) and
one nearly strongly synchronized system ($C=0.85$).

On the other hand, the dashed square in the upper center of Figure \ref%
{figure2}\textbf{\ }encompasses the coupled systems with coupling constants
(left to right) $C=0.55$, $0.65$, $0.45$, $0.40$, $0.50$ and $0.60$, i.e.,
all weakly synchronized systems ($0.50\leq C\leq 0.60$) and three nearly
weakly synchronized system ($C=0.40$, $0.45$, $0.65$).

Points located in the lower right corner of the entropy-complexity plane
correspond to systems with small coupling strengths.

The main conclusions from Figure \ref{figure2} can be summarized as follows.

\begin{itemize}
\item All transcript time series generated by strongly synchronized systems
(left cluster) have a lower complexity than the transcript time series
generated by weakly and nearly weakly synchronized systems (upper cluster).
The lowest statistical complexities (lower right corner) correspond to weakly coupled systems.

\item Also, transcript time series corresponding to weak synchronization
have a greater entropy than those corresponding to strong synchonization.
The transcript time series with the highest entropies correspond to weakly
coupled systems.
\end{itemize}

Therefore, the entropy-complexity plane can discriminate strong from weak
synchronization in the drive-response model (\ref{HenonMapX})-(\ref{HenonMap2}%
) parameterized by the coupling constant $C$. The systems are roughly
located along a parabola, the left and right branches corresponding to
systems with strong and weak couplings, respectively. Weakly synchronized
and nearly weakly synchronized systems are located around the top of the
parabola. However, while the points corresponding to strong synchronization are closely grouped, the points corresponding to weak synchronization are more scattered and mixed with points corresponding to nearly, but not strictly belonging to, weak synchronization. Therefore, the entropy-complexity plane is an accurate clustering tool for strong synchronization but not for weak synchronization. The joint asymptotic probability distribution of entropy and complexity has been studied by Silbernagel and Wei\ss\ \cite{Silbernagel2025}.

%%%%%%%%%%%%%%%%%%%%%%%%%%%%%%%%%%%%%%%%%%%%%%

\subsection{Results with transcript mutual information}
\label{sec73}

\begin{figure*}[tbh]
\begin{center}
\includegraphics[width=140mm]{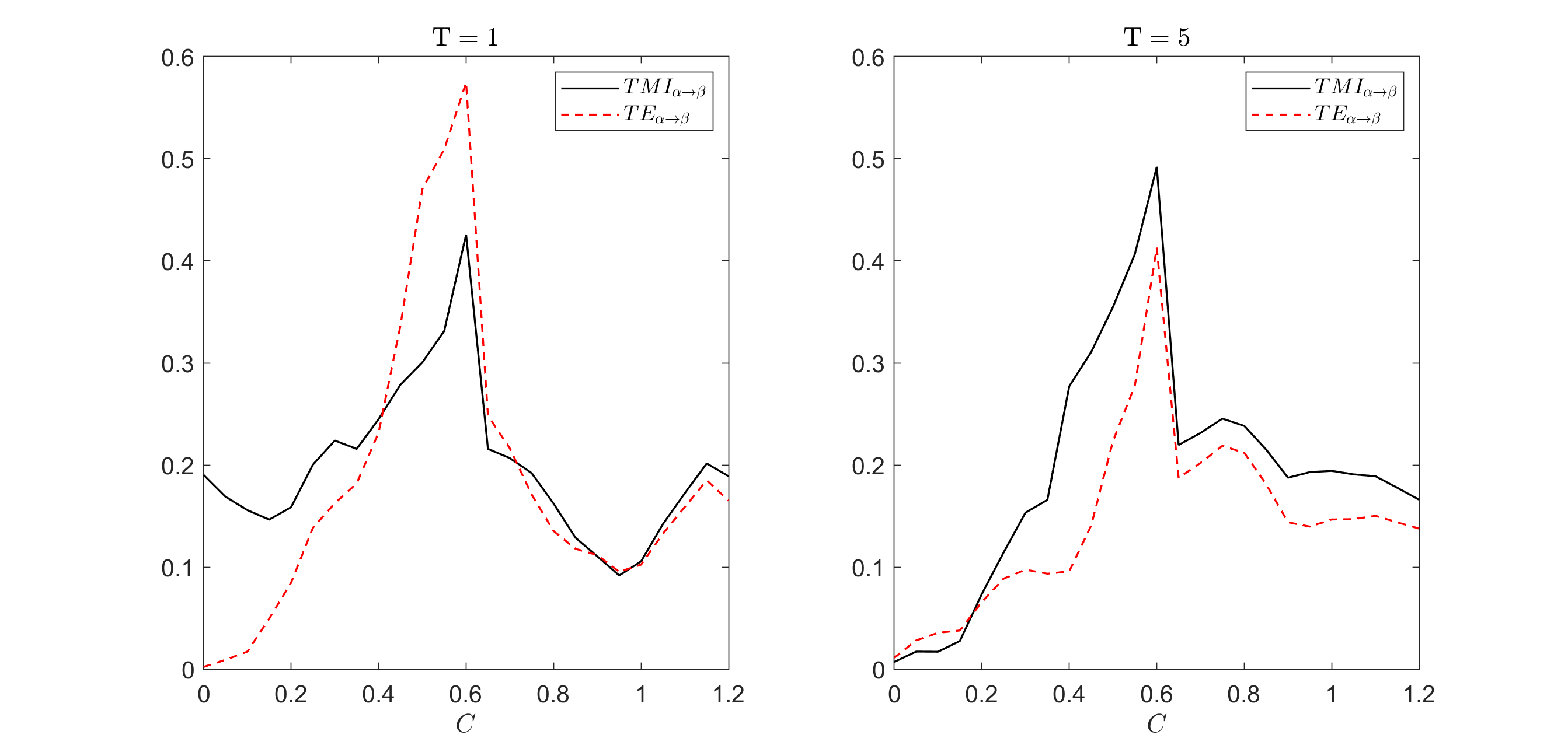}
\end{center}
\caption{Symbolic transfer entropy $\mathrm{TE}_{\mathrm{\protect\alpha }%
\rightarrow \mathrm{\protect\beta }}(1)=I(\mathbf{s}_{t+1};\mathbf{r}%
_{t}\left\vert \mathbf{s}_{t}\right) =:\mathrm{TE}_{\mathrm{\protect\alpha }%
\rightarrow \mathrm{\protect\beta }}$ and transcript mutual information $%
\mathrm{TMI}_{\mathrm{\protect\alpha }\rightarrow \mathrm{\protect\beta }%
}(1)=(T_{\mathbf{s}_{t+1},\mathbf{s}_{t}};T_{\mathbf{r}_{t},\mathbf{s}%
_{t}})=:\mathrm{TMI}_{\mathrm{\protect\alpha }\rightarrow \mathrm{\protect%
\beta }}$ as functions of the coupling strength $C$ for time lags $\textrm{T}=1$ (left panel) and $\textrm{T}=5$ (right panel). The positivity of the two
directionality indicators for the coupling delay $\Lambda =1$ in both panels suggests that
the process $\mathrm{\protect\alpha }$ is driving the process $\mathrm{%
\protect\beta }$. Also in both panels, the most visible feature of each curve is its global
maximum at weak synchronization ($C\simeq 0.60$), while strong
synchronization is not particularly signalized.}
\label{figure4}
\end{figure*}

Given the two coupled systems $X$ and $Y$ in equations (\ref{HenonMapX})-(%
\ref{HenonMap2}), and the observations $\mathrm{x}=(x_{t}^{(1)})_{1\leq
t\leq N}$ and $\mathrm{y}=(y_{t}^{(1)})_{1\leq t\leq N}$, the symbolic
transfer entropy \cite{Staniek2008} 
\begin{equation}
\mathrm{TE}_{\mathrm{\alpha }\rightarrow \mathrm{\beta }}(\Lambda )=I(%
\mathbf{s}_{t+\Lambda };\mathbf{r}_{t}\left\vert \mathbf{s}_{t}\right)
\label{TE L=3}
\end{equation}%
is a measure of the information transferred from $X$ to $Y$ based on the
ordinal representations $\mathrm{\alpha }=(\mathbf{r}_{t})_{1\leq t\leq
N-L+1}$ and $\mathrm{\beta }=(\mathbf{s}_{t})_{1\leq t\leq N-L+1}$ of $%
\mathrm{x}$ and $\mathrm{y}$, where here $\mathbf{r}_{t},\mathbf{s}_{t}\in 
\mathrm{Sym}(3)$.

As explained in Section \ref{sec54}, the transcript mutual information
\begin{equation}
\mathrm{TMI}_{\mathrm{\alpha }\rightarrow \mathrm{\beta }}(\Lambda )=I(T_{%
\mathbf{s}_{t+\Lambda },\mathbf{s}_{t}};T_{\mathbf{r}_{t},\mathbf{s}_{t}})
\label{TMI L=3}
\end{equation}%
coincides with $\mathrm{TE}_{\mathrm{\alpha }\rightarrow \mathrm{\beta }%
}(\Lambda )$ when two conditions are met: (i) $H(\mathbf{s}_{t})\leq H(\mathbf{r}%
_{t})$ and (ii) $C($\textrm{$\beta $}$_{\Lambda },\mathrm{\beta },\mathrm{\alpha })=0$, where $C(\cdot )$ is the coupling complexity coefficient of dimension
3 defined in equation (\ref{CCC}). Although condition (ii) is often only approximately met, numerical simulations and analysis with real data indicate that, in general, $C($\textrm{$\beta $}$_{\Lambda },\mathrm{\beta },\mathrm{\alpha }) \to 0$ and, hence, $\mathrm{TMI}_{\mathrm{\alpha }\rightarrow \mathrm{\beta }}(\Lambda ) \to \mathrm{TE}_{\mathrm{\alpha }\rightarrow \mathrm{\beta }}(\Lambda )$ as the time lag $\textrm{T}$ increases.

On the other hand, the possibility of calculating or
estimating a 3-dimensional information directionality indicator (symbolic
transfer entropy) with a 2-dimensional quantity (transcript mutual
information) is a very useful result in practice, where time series use to
be short and probabilities have to be estimated by frequencies. Considering
that, when determining the information directionality, we are usually only interested
in the sign of $\mathrm{TE}_{\mathrm{\alpha }\rightarrow \mathrm{\beta }%
}(\Lambda )$, the approximation error $\varepsilon =\left\vert \mathrm{TE}_{%
\mathrm{\alpha }\rightarrow \mathrm{\beta }}(\Lambda )-\mathrm{TMI}_{\mathrm{%
\alpha }\rightarrow \mathrm{\beta }}(\Lambda )\right\vert $ is
inconsequential as long as $\mathrm{TE}_{\mathrm{\alpha }\rightarrow \mathrm{%
\beta }}(\Lambda )$ and $\mathrm{TMI}_{\mathrm{\alpha }\rightarrow \mathrm{%
\beta }}(\Lambda )$ have the same sign, which translates into the assumption 
$\varepsilon <\left\vert \mathrm{TE}_{\mathrm{\alpha }\rightarrow \mathrm{%
\beta }}(\Lambda )\right\vert $.

Figure \ref{figure4} shows $\mathrm{TE}_{\mathrm{\alpha }\rightarrow \mathrm{%
\beta }}:=\mathrm{TE}_{\mathrm{\alpha }\rightarrow \mathrm{\beta }}(1)$ and $%
\mathrm{TMI}_{\mathrm{\alpha }\rightarrow \mathrm{\beta }}:=\mathrm{TMI}_{%
\mathrm{\alpha }\rightarrow \mathrm{\beta }}(1)$ as functions of the
coupling strength $C$ for time lags $\textrm{T}=1$ (left panel) and $\textrm{T}=5$ (right panel). As expected, the approximation error $\varepsilon$ is lower for $\textrm{T}=5$, in particular, $\mathrm{TMI}_{\mathrm{%
\alpha }\rightarrow \mathrm{\beta }}(1)=0$ for $C=0$ in that case. The main conclusions from Figure \ref{figure4} can be summarized as follows.

\begin{itemize}
\item The signs of $\mathrm{TE}_{\mathrm{\alpha }\rightarrow \mathrm{\beta }%
} $ and $\mathrm{TMI}_{\mathrm{\alpha }\rightarrow \mathrm{\beta }}$
coincide and are positive for $C>0$, as they should. As shown in the right panel, the deviation of $\mathrm{TMI}_{\mathrm{\alpha }\rightarrow 
\mathrm{\beta }}$ from $\mathrm{TE}_{\mathrm{\alpha }\rightarrow \mathrm{%
\beta }}$, especially for $C\leq 0.60$, can be reduced by taking larger time
lags in the time series $\mathrm{x}$ and $\mathrm{y}$.

\item Again, we observe a kind of complementary behavior regarding the two
synchronization states. For weak synchronization, both directionality
indicators have an absolute maximum, while their magnitudes in the strong
synchronization state are comparable to other dynamics with small couplings.
\end{itemize}

We conclude that the performance of the transcript mutual information $%
\mathrm{TMI}_{\mathrm{\alpha }\rightarrow \mathrm{\beta }}$ is satisfactory
to detect the driver, but unsatisfactory when it comes to detecting
generalized synchronization.

%%%%%%%%%%%%%%%%%%%%%%%%%%%%%%%%%%%%%%%%%%%%

\subsection{Results with transcript order classes}
\label{sec72}

\begin{figure*}[tbh]
\begin{center}
\includegraphics[width=140mm]{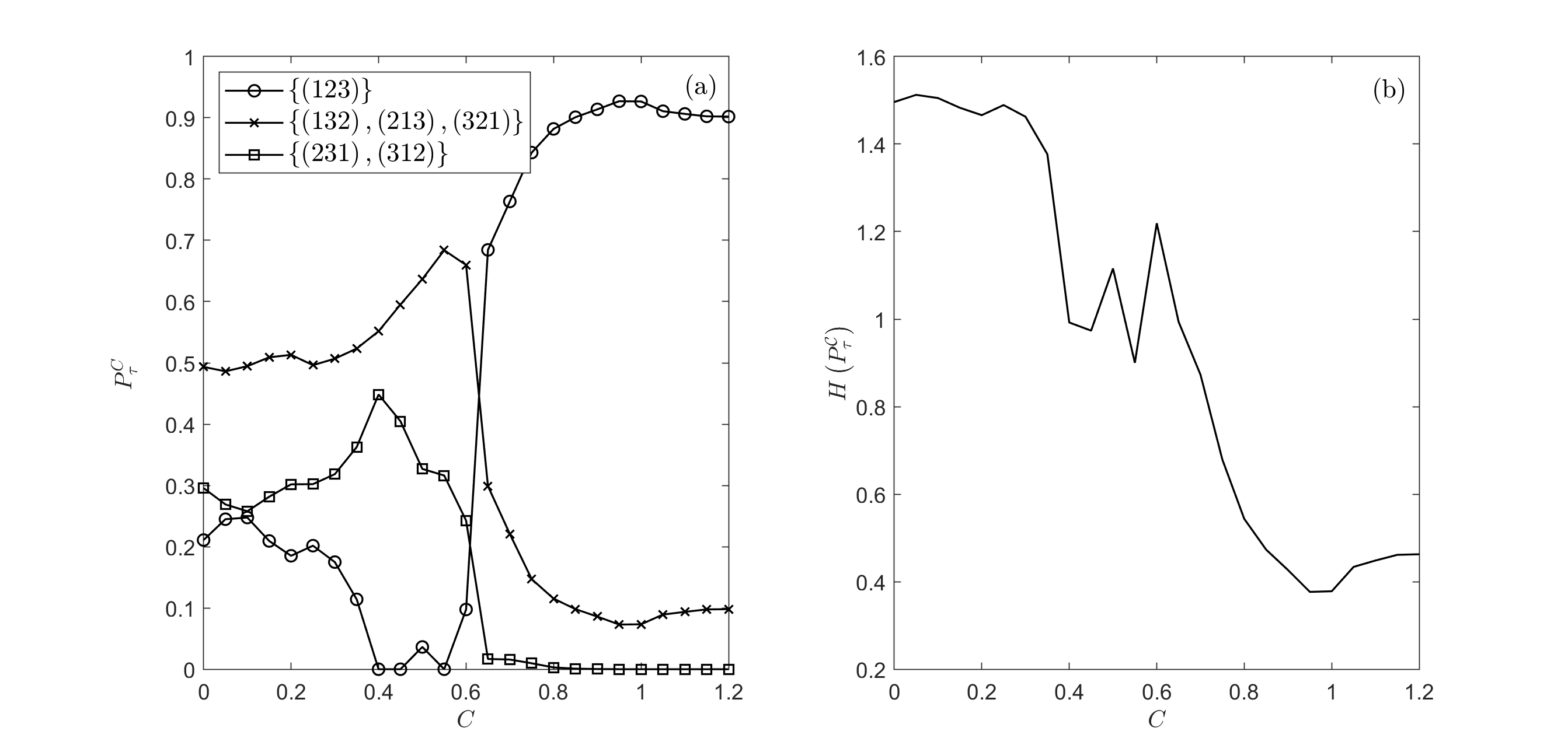}
\end{center}
\caption{(a) Probabilities of the transcript order classes (see inset) as
function of the coupling strength $C$. The probabilities of the lowest and
greatest order classes, $\mathcal{C}_{1}$ and $\mathcal{C}_{3}$, at the two
synchronization regimes are correlative in the following sense: $\mathcal{C}%
_{1}=\emptyset $ for weak synchronization while $\mathcal{C}_{3}=$ $%
\emptyset $ for strong synchronization. (b) Entropy of the order class
probability distribution $P_{\protect\tau }^{\mathcal{C}}$, equation (%
\protect\ref{p_C_m}). In this case weak synchronization is characterized by
a local minimum at $C=0.55$, while strong synchronization is characterized
by an absolute minimum at $C=1$.}
\label{figure3}
\end{figure*}

In Amig\'{o} and Dale \cite{Amigo2025}, we found a kind of correspondent
behavior of the forbidden distances with respect to the weak and strong
synchronizations regimes, namely: for $\mathbf{r},\mathbf{s}\in \mathrm{Sym}%
(L)$ and $L=4,5$, short distances ($d_{K}(\mathbf{r},\mathbf{s})=0,1,...$)
are forbidden in the weak regime ($C=0.55$), whereas long distances ($d_{K}(%
\mathbf{r},\mathbf{s})=d_{K,\max },d_{K,\max }-1,...$) are forbidden in the
strong regime ($C=1.10$). This correspondence seems to be reflected also in
Figure \ref{figure3}(a), this time with the probabilities of the transcript
order classes of the group $\mathrm{Sym}(3)$: 
\begin{equation}
\mathcal{C}_{1}=\{123\},\;\mathcal{C}_{2}=\{132,213,321\},\;\mathcal{C}%
_{3}=\{231,312\}.  \label{order classes L=3}
\end{equation}%
To be more specific:

\begin{itemize}
\item For $C=0.40,0.45,0.55$, the transcript $\mathbf{t}_{t}=123$ is
forbidden (that is, $\mathcal{C}_{1}=\{123\}=\emptyset $), which means that,
for weakly synchronized (and nearly synchronized) systems, $\mathbf{s}%
_{t}\neq \mathbf{r}_{t}$ for all $t\geq 0$, i.e., the ordinal $3$-patterns
of the driver and responder never coincide. For $C=0.50,0.60$, the
probability of $\mathbf{t}_{t}=123$ is not strictly $0$ but very small.

\item For $C\gtrsim 0.85$, $p_{\tau }(123)\gtrsim 0.9$, meaning that, for
strongly synchronized systems, the ordinal $3$-patterns of the driver and
responder coincide more than $90\%$ of the time. Furthermore, the transcript 
$321$ is forbidden in the same parametric interval (not shown), hence, if $%
\mathbf{r}_{t}=k_{1}k_{2}k_{3}$ then $\mathbf{s}_{t}\neq k_{3}k_{2}k_{1}$.

\item For $C\geq 0.95$, the order class $\mathcal{C}_{3}=\{231,312\}$ is
empty, which means that, in the regime of strong synchronization, $\mathbf{r}%
_{t}\in \mathcal{C}_{1}$ implies $\mathbf{s}_{t}\in \mathcal{C}_{1}\mathcal{%
\cup C}_{2}$, (ii) $\mathbf{r}_{t}\in \mathcal{C}_{2}$ implies $\mathbf{s}%
_{t}\in \mathcal{C}_{1}\mathcal{\cup C}_{3}$, and (iii) $\mathbf{r}_{t}\in 
\mathcal{C}_{3}$ implies $\mathbf{s}_{t}\in \mathcal{C}_{2}$.
\end{itemize}

In sum, weak and strong generalized synchronization impose restrictions on
the joint probabilities $p(\mathbf{r}_{t},\mathbf{s}_{t})$ that can be
detected at the coarse scale of transcript order classes of $\textrm{Sym}(3)$.

The entropy of the probability distribution of the transcript order classes
given in Figure \ref{figure3}(a) is depicted in Figure \ref{figure3}(b). As
in Figure \ref{figure2}, we see again that strong synchronization is
characterized by the global lowest transcript entropy, while transcript entropy has a local minimum at weak synchronization.

\begin{remark} \label{Remark order=distance}
  Let us also point out that there is a direct correspondence between
transcript order classes and Cayley distance when $\mathcal{G}=\mathrm{Sym}%
(3)$. According to equation (\ref{order classes L=3}) and the distance
matrix for $d_{C}(\mathbf{r},\mathbf{s})$ in equation (\ref{d_C(Sym(3))}), it holds%
\begin{equation}
d_{C}(\mathbf{r},\mathbf{s})=d_{C}(\mathbf{e},\,T_{\mathbf{r},\mathbf{s}%
})=n\in \{0,1,2\}\Leftrightarrow T_{\mathbf{r},\mathbf{s}}\in \mathcal{C}%
_{n+1},  \label{order = distance}
\end{equation}%
where $\mathbf{e}=123$, hence
\begin{equation}
\mathrm{ord}(T_{\mathbf{r},\mathbf{s}})-1 = d_{C}(\mathbf{r},\mathbf{s}).
\label{order = distance2}
\end{equation}
In other words, the integer $\mathrm{ord}(T_{\mathbf{r},\mathbf{s}})-1$ is actually a distance for the group Sym(3), 
namely, the Cayley distance between $\mathbf{r}$ and $\mathbf{s}$.  
\end{remark}
 For the
Kendall distance $d_{K}(\mathbf{r},\mathbf{s})$, the identification (\ref%
{order = distance}) is also correct except for $T_{\mathbf{r},\mathbf{s}%
}=321 $ (see equation (\ref{d_K(Sym(3))})). This explains the correlation
between forbidden distances and empty order classes mentioned above, which
is exact in the case of Cayley distances.

%%%%%%%%%%%%%%%%%%%%%%%%%%%%%%%%%%%%%%%%%%%%%%%%%%%

\subsection{Results with algebraic and entropic distances}
\label{sec74}

\begin{figure*}[tbh]
\begin{center}
\includegraphics[width=130mm]{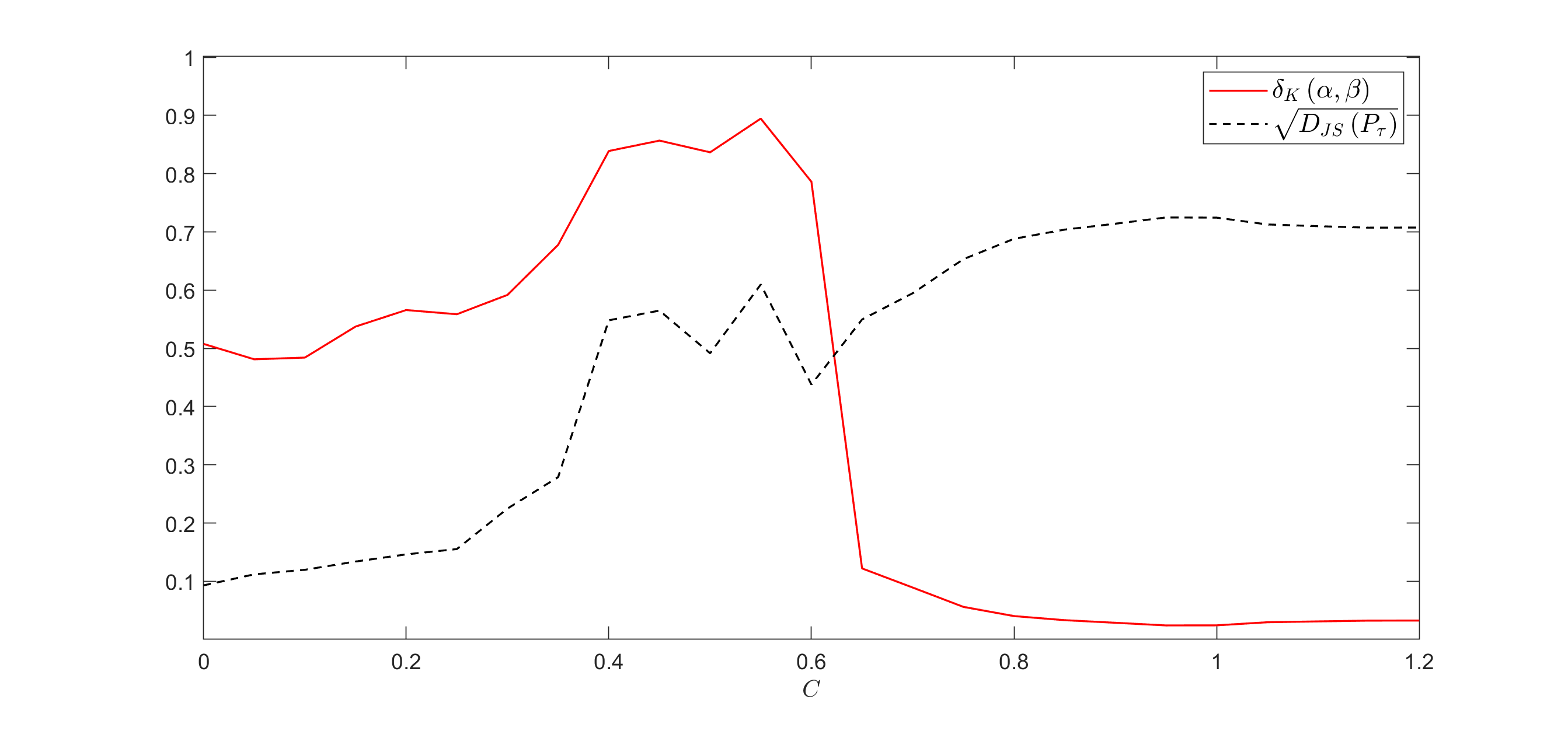}
\end{center}
\caption{The transcript-based similarity distance $\protect\delta _{K}(%
\mathrm{\protect\alpha },\mathrm{\protect\beta })$ and the entropy-based,
normalized distance $D_{JS}(P_{\mathbf{\protect\tau }})^{1/2}$ versus the
coupling strength $C$. The metric $\protect\delta (\mathrm{\protect\alpha },%
\mathrm{\protect\beta })$ measures the average Kendall distance $d_{K}(%
\mathbf{r}_{t},\mathbf{s}_{t})$ between simultaneous ordinal patterns of
the driver and responder; $D_{JS}(P_{\mathbf{\protect\tau }})^{1/2}$
measures the similarity between the transcript probability distributions $P_{%
\mathbf{\protect\tau }}$ and the uniform distribution $U$. In doing so, $%
\protect\delta _{K}(\mathrm{\protect\alpha },\mathrm{\protect\beta })$
achieves a global maximum at weak synchronization ($C\simeq 0.55$) and a
flat global minimum at strong synchronization ($C\gtrsim 1.0$), which is a
clear signal in both cases.}
\label{figure5}
\end{figure*}

Finally, we consider two distances between the symbolic time series $\mathrm{%
\alpha }=(\mathbf{r}_{t})_{1\leq t\leq N-L+1}$ and $\mathrm{\beta }=(\mathbf{%
s}_{t})_{1\leq t\leq N-L+1}$ for each coupling constant $0<C\leq 1.2$, $%
\Delta C=0.05$. First, we take the mean value over time of the sequence of
normalized Kendall distances $d_{K}(\mathbf{r}_{t},\mathbf{s}_{t})$, 
\begin{equation*}
(d_{K}(\mathbf{r}_{t},\mathbf{s}_{t})/d_{K,\max })_{1\leq t\leq
N-2}=(\left\Vert T_{\mathbf{r}_{t},\mathbf{s}_{t}}\right\Vert _{K}/3)_{1\leq
t\leq N-2}\text{,}
\end{equation*}%
to obtain the \textit{similarity distance}%
\begin{equation}
\delta _{K}(\mathrm{\alpha },\mathrm{\beta })=\frac{1}{3(N-2)}%
\sum_{t=1}^{N-2}\left\Vert T_{\mathbf{r}_{t},\mathbf{s}_{t}}\right\Vert _{K}
\label{d_K average}
\end{equation}%
between $\mathrm{\alpha }$ and $\mathrm{\beta }$ for each $C$. So, $\delta
_{K}(\mathrm{\alpha },\mathrm{\beta })$ measures similarity in the following
sense: the smaller (resp. greater) $\delta _{K}(\mathrm{\alpha },\mathrm{%
\beta })$ is, the more similar (resp. dissimilar) $\mathrm{\alpha }$ and $%
\mathrm{\beta }$ will be, meaning that the symbols of the time series $%
\mathrm{\alpha }$ and $\mathrm{\beta }$ will be closer (resp. more distant)
on average, where distance is measured by $d_{K}$.

To compare the performance of the transcript-based metric $\delta _{K}(%
\mathrm{\alpha },\mathrm{\beta })$ with other metrics, we have selected the
square root of the normalized Jensen-Shannon divergence, $D_{JS}(P_{\mathbf{%
\tau }})^{1/2}$, equation (\ref{SC Results}), which, according to Section %
\ref{sec52}, is a distance in axiomatic sense between the probability
distribution $P_{\mathbf{\tau }}$ and the uniform distribution $U$.
Therefore, the smaller (resp. greater) $D_{JS}(P_{\mathbf{\tau }})^{1/2}$
is, the more similar (resp. dissimilar) $P_{\mathbf{\tau }}$ and $U$ will
be. Since $C=0$ corresponds to the driver and responder being independent
(uncoupled), we expect that $D_{JS}(P_{\mathbf{\tau }})^{1/2}$ will have an
absolute minimum at $C=0$.

Figure \ref{figure5} depicts $\delta _{K}(\mathrm{\alpha },\mathrm{\beta })$
and $D_{JS}(P_{\mathbf{\tau }})^{1/2}$ as functions of $C$. The salient
points are the following.

\begin{itemize}
\item The algebraic similarity distance $\delta _{K}(\mathrm{\alpha },\mathrm{\beta })$
achieves its highest value (maximum dissimilarity between the time series $%
\mathrm{\alpha }$ and $\mathrm{\beta }$) at weak synchronization ($C=0.55$),
and its lowest values (maximum similarity between $\mathrm{\alpha }$ and $%
\mathrm{\beta }$) at strong synchronization. The transition is monotone and
steep.

\item On the other hand, the entropic Jensen-Shannon distance $D_{JS}(P_{\mathbf{\tau 
}})^{1/2}$ achieves a local maximum (i.e., a locally high dissimilarity
between the probability distributions $P_{\tau }$ and $U$) at weak
synchronization ($C=0.55$), and a flat global maximum (i.e., maximum
dissimilarity between $P_{\tau }$ and $U$) at strong synchronization.
\end{itemize}

The results for $\delta _{K}(\mathrm{\alpha },\mathrm{\beta })$ agree with
those obtained in Amig\'{o} and Dale \cite{Amigo2025} using empirical probability distributions for the Kendall distance and $L=4,5$, which
state that short distances are forbidden in the first case, while long
distances are forbidden in the second one. This disparity (also reflected in
the transcript order classes, Section \ref{sec72}) confirms that the nature
of the weak and strong synchronization regimes is different.

In sum, both the algebraic distance $\delta _{K}(\mathrm{\alpha },\mathrm{\beta })$ and the entropic distance $D_{JS}(P_{\mathbf{\tau }})^{1/2}$ are sensitive to the occurrence of synchronization, however, $\delta _{K}(\mathrm{\alpha },\mathrm{\beta })$ does so more clearly and distinctively. Actually, judging by the results obtained in the previous numerical simulations, $\delta _{K}(\mathrm{\alpha },\mathrm{\beta })$ is the best performer of all transcript-based tools considered in Section \ref{sec7}.  

%%%%%%%%%%%%%%%%%%%%%%%%%%%%%%%%%%%%%%%%%%%%%%%%%%%%%

\section{Conclusions}

\label{sec8}

This study discusses the role of transcripts in algebraic representations of time series, introducing a novel transcript-based tool for the analysis of coupled time series in the same representation. Remember that the transcript $T(a,b)$ from an element $a$ of a finite group $(\mathcal{G},\cdot )$ to another group element $b$, was defined in Section \ref{sec4} as $T(a,b)=b\cdot a^{-1}=:T_{a,b}$.

Regarding the objective (1) in the Introduction (overview of transcripts and their applications), we set the necessary framework in Sections (\ref{sec2})-(\ref%
{sec4}) and revisited a selection of transcript-based tools in Section \ref%
{sec5}, along with a representative sample of published applications. In
particular, we reminded the relationship between transcripts and the Cayley
and Kendall distances:\ $d_{C,K}(a,b)=\left\Vert T_{a,b}\right\Vert _{C,K}$
if $\mathcal{G}=(\mathrm{Sym}(L),\ast )$, equation (\ref{norm2}). In Section %
\ref{sec62} we showed that the distances $d_{C,K}$ can be \textquotedblleft
transported\textquotedblright\ to distances $D_{C,K}^{(\Phi )}$ on any other
group $\mathcal{G}$ via Cayley's isomorphism $\Phi :\mathcal{G}\rightarrow 
\mathrm{Sym}(\mathcal{G})$ (Theorem \ref{ThmCayley}). Furthermore, Cayley's
isomorphism can be implemented using transcripts (Section \ref{sec4}).

When using Cayley's isomorphism and equation (\ref{dist(a,b)}) to define the
distance $D_{C,K}^{(\Phi )}$ in a group $\mathcal{G}\neq \mathrm{Sym}(L)$,
we are encoding the $\left\vert \mathcal{G}\right\vert $ elements of $%
\mathcal{G}$ as permutations on $\mathcal{G}$. 
According to Cicirello \cite%
{Cicirello2019}, there are efficient algorithms (such as the bubble-sort
algorithm) that compute $D_{C}^{(\Phi )}$ in time $O(\left\vert \mathcal{G}%
\right\vert )$, and $D_{K}^{(\Phi )}$ in time $O(\left\vert \mathcal{G}%
\right\vert \log \left\vert \mathcal{G}\right\vert )$. Therefore, the
embedding of $\mathcal{G}$ into $\mathrm{Sym}(\mathcal{G})$ to endow $%
\mathcal{G}$ with the distance $D_{C,K}^{(\Phi )}$ is a computationally
competitive method for the low cardinality alphabets used in practice. For
example, the calculation of $(D_{K}^{(\Phi
)}(a_{t},b_{t}))_{1\leq t\leq N}$ when $\mathcal{G=}\mathrm{Sym}(6)$ ($%
\left\vert \mathcal{G}\right\vert =6!=720$), $N=10,000$, and $\Phi (\mathbf{r%
})$ is the right translation $\mathbf{s}\mapsto \mathbf{r}\circ \mathbf{s}=%
\mathbf{s}\ast \mathbf{r}$, took $2.84$ seconds using a non-parallelized
algorithm running on a laptop computer \cite{Amigo2025}.

Regarding the objective (2) (benchmarking with numerical simulations), we chose
four traditional transcript-based tools (entropy-complexity plane,
transcript mutual information, order classes, and Jensen-Shannon distance),
as well as the Kendall distance, which has a higher discriminatory power
than the Cayley distance due to its larger range (except for $\mathrm{%
Sym}(2)$); see equations (\ref{Range d_C})-(\ref{Range d_K}). The numerical simulations
consisted of detecting generalized synchronization in the drive-response
system (\ref{HenonMapX})-(\ref{HenonMap2}) composed of two, non-identical H%
\'{e}non maps. The best performer turned out to be the novel similarity distance $%
\delta _{K}(\mathrm{\alpha },\mathrm{\beta })$, which is the mean Kendall
distance defined in equation (\ref{d_K average}).

Precisely, the objective (3) in the Introduction refers to the similarity distance. First of all, we remark its simplicity. As suggested in Section \ref{sec74}, its excellent performance in the numerical
simulations is related to the existence of forbidden transcripts (i.e., zero
probability transcripts) in the cases of weak and strong synchronization. These results support the idea that the novel similarity distance will also perform well in different contexts, confirming its utility in time series analysis. 

In retrospect, the main contributions in the previous sections can be summarized as follows.

\begin{description}

\item[(i)] Introduction of the novel similarity distance (\ref{d_K average}). We found that this transcript-based tool outperforms the others tested in Section \ref{sec7}.

\item[(ii)] Derivation of the right-invariance of $d_{C,K}$ using
transcripts (equations (\ref{T=T})-(\ref{right invariance})). In turn, the
right-invariance of $d_{C,K}$ amounts to the Cayley and Kendall distances
being norms of transcripts (equation (\ref{norm2})).

\item[(iii)] Transportation of the algebraic distances $d_{C,K}$ from $%
\mathrm{Sym}(L)$ to a general group $(\mathcal{G},\cdot )$ via Cayley's
embedding. This approach is practical for the small cardinalities $%
\left\vert \mathcal{G}\right\vert $ used in practice. Moreover, calculating
with $D_{C,K}^{(\Phi )}$ dispenses with the search for minimal descriptions
of group elements as products of generators.

\item[(iv)] Verification that the algebraic order in $\mathrm{Sym}(L)$,
sometimes used as a \textquotedblleft dissimilarity\textquotedblright\
measure\ in the literature \cite{Monetti2009}, is actually a distance in
strict sense after subtracting one for $L=3$ (since it coincides with the
Cayley distance); see equation (\ref{order = distance2}).

\item[(v)] Interpretation of the similarity distance's performance in the
numerical simulations ($\mathcal{G}=\mathrm{Sym}(3)$) in terms of forbidden
transcripts when the system (\ref{HenonMapX})-(\ref{HenonMap2}) is in weak
or strong generalized synchronization.

\item[(vi)] Existence of gaps in the transported distances $D_{C,K}^{(\Phi )}
$. These gaps must be identified to avoid misinterpreting them as a
characteristic of the underlying dynamics. This can be done by calculating
the row $(D_{C,K}^{(\Phi )}(e,c))_{c\in \mathcal{G}}$ of the distance
matrix, or using independent white noises for the source and target time
series.
\end{description}

Of the above contributions, which include both theoretical and practical matters, we may conclude that transcripts are useful tools in time series analysis. In particular, the similarity distance results in our numerical simulations are very promising.    

The concept of transcript, as defined in equation (\ref{T(a,b)}), can be
generalized in different ways while preserving its instrumental properties
but, so far, no such generalizations have proven to be as handy and
practical as the original concept. This suggests that future research in the
theory and practice of algebraic representations should focus on the
development of new tools and applications. Actually, the applications of
algebraic distances in the analysis of time series is the topic of current
research. In this regard, we hope that this paper will pave the way for
new algebra-based tools and applications in
group-valued time series analysis.

%%%%%%%%%%%%%%%%%%%%%%%%%%%%%%%%%%%%%%%%%

\begin{acknowledgments}
We are very grateful to our reviewers for their constructive criticism, which contributed decisively to improving the original manuscript.
\end{acknowledgments}

%%%%%%%%%%%%%%%%%%%%%%%%%%%%%%%%%%%%%%%%%%%%%%%%%%%%%%%%%%%%%%%%%%%%%%%%%%%%

\section*{Data Availability Statement}
The data that support the findings of this study are available from the corresponding author upon reasonable request.

%\appendix

%\section{Appendixes}

%\nocite{*}
%\bibliography{aipsamp}% Produces the bibliography via BibTeX.

\end{document}